\documentclass[11pt]{amsart}
\numberwithin{equation}{section}
\usepackage[english]{babel}
\usepackage[T1]{fontenc}
\usepackage[latin1]{inputenc}
\usepackage{indentfirst}
\usepackage{enumitem}
\usepackage{amsmath,amssymb, amsbsy}
\usepackage{comment}
\usepackage{amsfonts}
\usepackage{hyperref}
\usepackage{cleveref}
\usepackage{esint}
\usepackage{latexsym}
\usepackage{amsthm}
\usepackage[dvips]{graphicx}
\usepackage{xcolor}
\usepackage{tikz}
\usepackage{pgfplots}
\usetikzlibrary{intersections}
\usepackage{tikz-3dplot}
\usepackage[outline]{contour}
\DeclareGraphicsExtensions{.pdf,.png,.jpg,.eps}

\usepackage{amsmath}

\usepackage{bbm}

\usepackage[font=small,labelfont=bf]{caption}

\usepackage{amsthm}

\usepackage{hyperref}

\usepackage{xcolor}
\usepackage[latin1]{inputenc}
\usepackage[active]{srcltx}
\definecolor{Arancio}{cmyk}{0,0.61,0.87,0}

\definecolor{blus}{RGB}{0,102,204}

\usepackage{doi}

\newcommand{\brd}[1]{\mathbb{#1}}
\newcommand{\R}{\brd{R}}

\newcommand{\be}{\begin{equation}}
\newcommand{\ee}{\end{equation}}

\newtheorem{teo}{Theorem}[section]

\newtheorem{Lemma}[teo]{Lemma}
\newtheorem{Theorem}[teo]{Theorem}
\newtheorem{Proposition}[teo]{Proposition}
\theoremstyle{definition}

\newtheorem{remark}[teo]{Remark}

\renewcommand{\div}{\operatorname{div}}

\newcommand{\loc}{{\rm loc}}

\usepackage{amsmath}
\usepackage{tikz}
\usepackage{mathdots}
\usepackage{yhmath}
\usepackage{cancel}
\usepackage{color}
\usepackage{siunitx}
\usepackage{array}
\usepackage{multirow}
\usepackage{amssymb}
\usepackage{textcomp}
\usepackage{gensymb}
\usepackage{tabularx}
\usepackage{extarrows}
\usepackage{booktabs}
\usetikzlibrary{fadings}
\usetikzlibrary{patterns}
\usetikzlibrary{shadows.blur}
\usetikzlibrary{shapes}

\pgfplotsset{compat=1.18} 
\begin{document}



\title[Schauder estimates for ratios of elliptic solutions under zero set inclusion]{Schauder estimates for ratios of elliptic solutions\\ under zero set inclusion}

\author{Stefano Vita}

\address{Stefano Vita, Dipartimento di Matematica "F. Casorati", Universit\`a di Pavia, Via Ferrata 5, 27100, Pavia, Italy}
\email{stefano.vita@unipv.it}

\begin{abstract}
Let $u,v$ be two nontrivial functions that solve
$$
\mathrm{div}(A\nabla u)=\mathrm{div}(A\nabla v)=0\qquad\mathrm{in \ }B_r\subset\mathbb R^n,\quad n\geq2,\quad r>0,
$$
with $A$ symmetric, uniformly elliptic and Lipschitz continuous, and assume the zero set inclusion $Z(u)\subseteq Z(v)$. We prove that the ratio $w=v/u$ continuously extends, together with its gradient, across $Z(u)=\{u=0\}$. In particular, $w\in C^{1,\alpha}_\loc(B_{r})$ for any $\alpha\in(0,1)$. Furthermore, given any integer $k\geq2$ and $\alpha\in(0,1)$, if $A\in C^{k-1,\alpha}(B_r)$, then $w\in C^{k,\alpha}_\loc(B_{r})$. The estimates are uniform in classes of coefficients $A$ with prescribed ellipticity and regularity bounds, and in normalized classes of denominators $u$ with uniformly bounded doubling constant.
\end{abstract}

\maketitle


\section{Introduction}

The \emph{boundary Harnack principle} is a comparison principle for solutions of elliptic equations vanishing on a common portion of the boundary. In its classical form, if $u$ and $v$ are positive harmonic functions in a domain $\Omega$, vanish continuously on a portion of $\partial\Omega$, and the geometry of $\Omega$ is sufficiently regular, then the quotient $v/u$ remains controlled up to the boundary. This holds true in Lipschitz domains \cite{Kem72,Dah77,Anc78}, non-tangentially accessible domains \cite{JerKen82}, H\"older domains \cite{BanBasBur91}, and twisted H\"older domains of order greater than $1/2$ \cite{BasBur91}; see also \cite{DesSav20} for a more recent analytic approach.

When the boundary is more regular, one can ask for a correspondingly finer description of the quotient. This leads to the \emph{higher order boundary Harnack principle}. De Silva and Savin \cite{DesSav15} proved that, if $u$ is a positive harmonic function and $v$ is harmonic, both vanishing on a portion of the boundary of a $C^{k,\alpha}$ domain with $k\geq1$, then $v/u\in C^{k,\alpha}$ up to that portion of the boundary. Their proof is based on the approximation of $v$ by functions of the form $\pi u$, where $\pi$ is a polynomial, together with an improvement-of-flatness argument, see also \cite{TerTorVit24} for an approach through degenerate elliptic equations, \cite{JeoVit24,DonJeoVit24} for Dini-type boundary regularity and \cite{DesSav16} for slit domains.

Then, in \cite{LogMal15,LogMal16}, Logunov and Malinnikova started studying regularity properties of the ratio $v/u$ of two harmonic functions sharing their zero sets $Z(u)=Z(v)$, see also \cite{Man14}. They prove that $v/u$ is real analytic. This can be regarded as a higher order boundary Harnack principle across the zero set of the denominator $Z(u)=\{u=0\}$. Their approach strongly involves the real analyticity of $u,v$, polynomial expansions, and local division results for polynomials. In \cite{LinLin22,TerTorVit26a,TerTorVit26b,Fio26}, this study was further explored by considering the ratio of solutions to uniformly elliptic equations with variable coefficients, under the zero set inclusion $Z(u)\subseteq Z(v)$. The presence of variable coefficients does not enable, in general, the use of real analyticity techniques.

Let us consider the problem 
\begin{equation}\label{eq:class1}
L_Au=L_Av=0,\qquad Z(u)\subseteq Z(v)\qquad\mathrm{in \ }B_2\subset\R^n, \quad n\geq2,
\end{equation}
where $L_A\cdot=\mathrm{div}(A\nabla\cdot)$. Here $u,v\in H^1(B_2)$ are nontrivial real-valued weak solutions. The variable coefficients matrix is symmetric and uniformly elliptic with ellipticity constants $0<\lambda\leq\Lambda$; that is,
\begin{equation}\label{eq:class2}
A=(a_{ij})_{i,j=1,\dots,n},\qquad A=A^T,\qquad \lambda |\xi|^2\le A(x)\xi\cdot\xi\le\Lambda |\xi|^2 \qquad\forall \, x\in B_2, \quad \forall \, \xi\in\R^n.
\end{equation}

The condition $Z(u)\subseteq Z(v)$ is the geometric constraint at the core of the problem. We assume that $A$ is Lipschitz continuous in $B_2$. In this regularity class, nontrivial solutions satisfy local quantitative unique continuation estimates and, in particular, local doubling inequalities, see \cite{GarLin86}. Throughout the paper, we regard the doubling constant $D\geq 2^n$ as an additional quantitative parameter and we assume that
\begin{equation}\label{eq:doubling}
 \int_{B_{2r}(x_0)}u^2 \le D\int_{B_r(x_0)}u^2, \qquad \forall \, x_0\in B_{3/2},\quad \forall \, 0<r\le 1/4.
\end{equation}

In \cite{LinLin22}, the authors prove H\"older continuity of the ratio
with an implicit exponent $0<\alpha^*<1$ depending on the dimension, the ellipticity, coefficient regularity bounds and a frequency bound for the denominator (the latter is equivalent to a uniform control on the doubling constant). Their argument establishes corkscrew conditions and modified Harnack
chains for nodal domains, in the spirit of the theory of Jerison and Kenig for NTA domains \cite{JerKen82}.

The works \cite{TerTorVit26a,TerTorVit26b} develop another approach, based on equations whose ellipticity degenerates on $Z(u)$. In this formulation the quotient $w=v/u$ is treated as a weighted weak solution, in $H^1_{\loc}(B_2,u^2dx)$, of
\begin{equation}\label{eq:ratioequation}
\mathrm{div}(u^2A\nabla w)=0\qquad\mathrm{in \ }B_2.
\end{equation}
This approach helped to improve the $C^{0,\alpha^*}$-result in dimension $n=2$, providing optimal $C^{1,\alpha}$-regularity of the ratio for any $0<\alpha<1$.

In this paper, we establish quantitative Schauder estimates for the ratio $w=v/u$ across the entire nodal set of the denominator, including its singular points, in every dimension $n\ge2$.

Our first result concerns H\"older continuity of $w$.

\begin{Theorem}\label{thm:holder}
Assume \eqref{eq:class1}, \eqref{eq:class2}, \eqref{eq:doubling} and
$[A]_{C^{0,1}(B_2)}\le L$. Then $w=v/u$ has a unique continuous extension to $B_1$. Moreover, for every $0<\alpha<1$ there exists a constant $C>0$ depending only on $n,\lambda,\Lambda,L,D,\alpha$ such that
\begin{equation}\label{eq:holder-conclusion}
 \|w\|_{C^{0,\alpha}(B_1)}\le C\frac{\|v\|_{L^2(B_2)}}{\|u\|_{L^2(B_1)}}.
\end{equation}
\end{Theorem}
This first result already improves, for general Lipschitz coefficients, the H\"older continuity result in \cite{LinLin22}, which comes with the implicit small exponent $\alpha^*$. It also removes the a priori H\"older regularity assumption in \cite{TerTorVit26a}. Then, one can actually prove more.

\begin{Theorem}\label{thm:lip}
Assume \eqref{eq:class1}, \eqref{eq:class2}, \eqref{eq:doubling} and
$[A]_{C^{0,1}(B_2)}\le L$. Then, for every $0<\alpha<1$ there exists a constant $C>0$ depending only on $n,\lambda,\Lambda,L,D,\alpha$ such that
\begin{equation}\label{eq:lip-conclusion}
 \|w\|_{C^{1,\alpha}(B_1)}\le C\frac{\|v\|_{L^2(B_2)}}{\|u\|_{L^2(B_1)}}.
\end{equation}
\end{Theorem}
This second result, in case of general Lipschitz coefficients, extends the previously known result in \cite{TerTorVit26b} to any dimension $n\geq2$.

Finally, as one considers more regular coefficients, a complete Schauder theory for the ratio can be proved.

\begin{Theorem}\label{thm:higher}
Assume \eqref{eq:class1}, \eqref{eq:class2}, \eqref{eq:doubling}, $k\in\mathbb N_{\geq2}$, $0<\alpha<1$, and $\|A\|_{C^{k-1,\alpha}(B_2)}\le L$.
Then there exists a constant $C>0$ depending only on $n,\lambda,\Lambda,L,D,k,\alpha$ such that
\begin{equation}\label{eq:higher-conclusion}
\|w\|_{C^{k,\alpha}(B_1)}\le C\frac{\|v\|_{L^2(B_2)}}{\|u\|_{L^2(B_1)}}.
\end{equation}
\end{Theorem}

The constants in these estimates depend only on the prescribed ellipticity and coefficient bounds, the doubling constant, and the regularity indices. Thus the estimates remain uniform when the denominator varies within this class, even if the local geometry of its nodal set changes.

Notice also that Theorems \ref{thm:lip} and \ref{thm:higher} are optimal, in the sense that one can easily provide counterexamples to any better H\"older regularity. Moreover, when $A$ is real analytic, the argument of \cite{LogMal15,LogMal16} extends to show that $w$ is real analytic. In this sense, the present results complete the regularity picture for boundary Harnack principles on nodal domains.

\subsection{Strategy of the argument}
We would like to briefly discuss the strategy underlying the arguments of this paper. The techniques developed here are, in spirit, much closer to \cite{DesSav15,LogMal15,LogMal16} than to \cite{LinLin22} or \cite{TerTorVit26a,TerTorVit26b}. Moreover, the proofs developed here do not require any of the previously known regularity results contained in the latter papers.

The approach of \cite{LinLin22} treats the nodal set $Z(u)$ essentially as a fixed, possibly rough, boundary and develops geometric boundary Harnack arguments inside the nodal domains of $u$; that is, connected components of $\{u\neq0\}$. This viewpoint is well suited to proving H\"older continuity of the quotient, but becomes less natural when one seeks higher regularity. Indeed, $Z(u)$ is not an actual boundary for the original equations: both $u$ and $v$ solve the same elliptic equation across $Z(u)$, and the information contained in this transmission across the nodal set is lost if one works independently on the different nodal domains.

A different viewpoint is provided by the degenerate equation
\eqref{eq:ratioequation}, used in \cite{TerTorVit26a,TerTorVit26b}. Its main advantage is that it is naturally posed on the whole ball and therefore retains the global character of the quotient across $Z(u)$. On the other hand, the geometry of the nodal set is encoded in the degenerating weight $u^2$. In particular, near the singular set $S(u)=\{u=0,\ \nabla u=0\}$, the weight may vanish to high order - with an upper bound when the doubling constant $D$ is fixed - and its degeneracy may occur along sets with a complicated stratified structure. This makes it difficult to obtain higher order estimates directly from the degenerate equation alone.

As far as the nodal set $Z(u)$ is concerned, the arguments developed here use essentially only the doubling property of the weighted measure
$$
d\mu=u^2\,dx.
$$
The geometry of $Z(u)$, including its stratified structure and the possible complexity of its singular set, is largely invisible to the arguments contained in our proofs.

For later notational convenience, we use $d\in\mathbb N$ in place of $k$
for the regularity index. The proof of $C^{d,\alpha}$ regularity is based
on a \emph{best polynomial approximation} argument: at every point $x_0$ and scale
$r>0$, we approximate $w$ by a polynomial $\pi_{x_0,r}\in\mathcal P_d$
which best fits $w$ in the weighted space $L^2(d\mu)$. Equivalently, this
amounts to approximating $v$ by $\pi_{x_0,r}u$ in the usual $L^2$ norm.
The corresponding normalized approximation error, the $d$th-order excess, is
$$
\mathcal E_d(x_0,r)=\inf_{\pi\in\mathcal P_d}\left(\fint_{B_r(x_0)}|w-\pi|^2\,d\mu\right)^{1/2}=\inf_{\pi\in\mathcal P_d}\left(\frac{\int_{B_r(x_0)}|v-\pi u|^2}{\int_{B_r(x_0)}u^2}\right)^{1/2},
$$
where $\mathcal P_d$ denotes the space of real polynomials of degree at
most $d$. Here $\fint_B f\,d\mu=\mu(B)^{-1}\int_Bf\,d\mu$.

As we have already remarked, the approximation of $v$ by $u$ times a
polynomial is central both to the higher order boundary Harnack argument
of De Silva and Savin \cite{DesSav15} and to the analytic division method
of Logunov and Malinnikova \cite{LogMal15,LogMal16}. There are, however,
two substantial differences in the present setting. First, the vanishing
of $u$ along $Z(u)$ is not necessarily linear, nor is there in general a
distinguished normal direction along which $u$ has a nonvanishing
derivative of fixed sign, as in \cite{DesSav15}. Second, in contrast with
the analytic setting of \cite{LogMal15,LogMal16}, we do not have convergent
power series at our disposal. Polynomial approximation therefore produces
genuine remainder terms, whose control is one of the main difficulties of
the argument. 

Let us now summarize the main steps of the regularity argument:

\begin{itemize}
\item[$1)$] existence and uniqueness of the best polynomial fit, together with quantitative control of its coefficients and their variation across scales;
\item[$2)$] decay of the corresponding excess with the appropriate power of the radius;
\item[$3)$] a Campanato-type argument \cite{Cam63,Cam64} which identifies the limiting polynomial coefficients with the derivatives of $w$ and yields the desired estimates in H\"older norms.
\end{itemize}

The core of the argument is point $2)$, which is proved by a blow-up and
compactness procedure in the spirit of classical excess-decay arguments;
see \cite{GiuMir68,GiaNec79,Eva86,EvaGar87}. To establish the decay estimate
\begin{equation}\label{eq:excess-bound-intro}
\mathcal E_d(x_0,r)\leq Cr^{d+\alpha},
\end{equation}
we normalize the data and argue by contradiction. We select centers $x_k$, scales $r_k\to0$, and excesses $\sigma_k:=\mathcal E_d(x_k,r_k)$ such that
$$
M_k=\frac{\sigma_k}{r_k^{d+\alpha}}\to\infty,
$$
while the excesses at all larger scales, up to a fixed radius $r_0$, remain controlled in terms of $M_k$. Set the blow-up sequences
$$
\kappa_k^2=\int_{B_1}u_k^2(x_k+r_kx),\qquad
 \tilde u_k(x)=\frac{u_k(x_k+r_kx)}{\kappa_k},\qquad
 \tilde v_k(x)=\frac{(v_k-\pi_{x_k,r_k}u_k)(x_k+r_kx)}{\kappa_k\sigma_k}.
$$
Notice that $Z(\tilde u_k)\subseteq Z(\tilde v_k)$. By the minimizing properties of $\pi_{x_k,r_k}$,
$$
\|\tilde u_k\|_{L^2(B_1)}=\|\tilde v_k\|_{L^2(B_1)}=1,
$$
and $\tilde v_k$ is orthogonal to $\tilde u_k\mathcal P_d$ in $L^2(B_1)$. This orthogonality condition will play a central role in the contradiction argument. It is precisely the Euler-Lagrange condition associated with the choice of $\pi_{x_k,r_k}$ as the best polynomial fit in the weighted $L^2$ norm.

Then, comparison of the best fits at consecutive scales yields, on every fixed ball $B_R$, $R\geq1$, for all sufficiently large $k$ depending on $R$,
$$
\|\tilde v_k\|_{L^2(B_R)}\leq CR^{d+\alpha}\|\tilde u_k\|_{L^2(B_R)},\qquad \int_{B_R}\tilde u_k^2 \leq DR^{\log_2D}.
$$
The constants in these growth bounds are independent of $R$ and $k$.
The equation satisfied by the numerator remainder depends on the degree
of the polynomial fit. Constant subtraction preserves the original
elliptic equation. After affine subtraction, a suitable frozen-coefficient
replacement leads in the limit to a biharmonic equation. In the
higher-degree case, one retains the Taylor polynomial of the coefficients
and introduces finitely many correction terms; the corresponding limiting
remainder is polyharmonic of finite order. Smoothness, Caccioppoli
estimates and Sobolev embedding then provide the required compactness for
these frozen-coefficient polyharmonic remainders. Together with the
compactness of the denominator, this yields convergence to some limiting profiles
$$
\tilde u_k\to U,\quad\tilde v_k\to V,
$$
locally uniformly and strongly in $L^2_{\loc}(\R^n)$. The rescaled coefficients converge to a constant uniformly elliptic matrix $A_0$. We first pass the nodal inclusion to these limits. Every zero of the nontrivial $L_{A_0}$-harmonic function $U$ is approached by zeros of $\tilde u_k$, because $U$ changes sign in every neighborhood of such a zero. At those points $\tilde v_k=0$, so local uniform convergence gives $Z(U)\subseteq Z(V)$.

After a fixed linear change of coordinates, we reduce the limiting equations satisfied by $U,V$ to the corresponding equations with identity coefficients satisfied by $\tilde U,\tilde V$. Explicitly, $S=A_0^{1/2}$, $\tilde U(y)=U(Sy)$ and $\tilde V(y)=V(Sy)$. In these coordinates, a polynomial integral-growth Liouville theorem, see Theorem \ref{thm:app-liouville}, implies that the limiting profiles are polynomials, say
$$
\tilde U=Q,\qquad \tilde V=P.
$$
Here $Q$ is harmonic, while $P$ is in general only polyharmonic. The relative growth estimate gives
$$
\deg P\leq \deg Q+d.
$$
Moreover, the linear change preserves the nodal inclusion, so that
$$
Z(Q)\subseteq Z(P).
$$

At this point, we can apply the main rigidity result, due to Murdoch \cite[Theorem 2 and Lemma 4]{Mur64}. The precise version below is also stated in \cite[Lemma 5.6]{TerTorVit26a}.

\begin{Lemma}[Murdoch's division lemma \cite{Mur64}]\label{lem:murdoch}
Let $Q$ be a nontrivial harmonic polynomial and $P$ be a polynomial with real coefficients in $\R^n$, such that $Z(Q)\subseteq Z(P)$. Then, there exists a polynomial $R$ such that $P=QR$.
\end{Lemma}

It is important to remark here that the above result requires harmonicity only of $Q$: the numerator $P$ can be any polynomial. This is why we can allow the limiting numerator remainder to be polyharmonic. Then, Murdoch's lemma yields
$$
P=Q R,\qquad\deg R\leq d.
$$
Transporting this identity back through the inverse linear change of coordinates, we obtain
$$
V=U\tilde R
$$
for some $\tilde R\in\mathcal P_d$. Indeed, $\tilde R=R\circ S^{-1}$, and composition with an invertible linear map preserves polynomial degree. The limiting orthogonality $V \perp U\mathcal P_d$ therefore contradicts the unit $L^2(B_1)$ normalization of $V$
$$
0=\int_{B_1}VU\tilde R=\int_{B_1}V^2=1.
$$

Finally, let us remark that the advantage of proving \eqref{eq:excess-bound-intro} before deriving the regularity estimates in H\"older norms is that the excess is well defined independently of any a priori regularity of the quotient. A contradiction argument based directly on the $C^{d,\alpha}$ seminorm, in the spirit of \cite{Sim97}, would instead require the desired regularity to be already available, at least at an a priori level. Such an a priori--a posteriori scheme is particularly inconvenient in the present setting, since no natural regularization procedure preserving both the quotient structure and the nodal set inclusion is available. The excess, instead, is well defined without any prior regularity of $w$, and its decay will itself produce the desired regularity.

\subsection{Structure of the paper}
Section \ref{sec:holder} proves Theorem \ref{thm:holder}. We first construct the best constant fits, then prove decay of the zeroth-order excess, and finally recover the continuous quotient and its H\"older estimate.
Section \ref{sec:first-order} follows the same steps with affine fits to prove Theorem \ref{thm:lip}; the additional ingredient is compactness of the amplified affine remainders through frozen-coefficient biharmonic replacements.
Section \ref{sec:higher-order} repeats this scheme with polynomial fits and proves Theorem \ref{thm:higher}. Here we keep the Taylor polynomial of the coefficients and use finitely many corrections to obtain polyharmonic replacements with a sufficiently small error.
Appendix \ref{app:polyharmonic} collects the second-order Schauder estimates used for the replacements, smoothness for powers of the Laplacian, the constant-coefficients Caccioppoli inequality and the polynomial Liouville theorem under integral growth for polyharmonic functions. We finally remark that some of the results in the paper could have been condensed into a single statement valid for every degree $d\in\mathbb N$. We have preferred to present the different orders separately, since the notation becomes increasingly involved as the degree grows. This allows the reader to become familiar with the main mechanism at the first levels, before facing the additional algebraic and notational complications of the general case.\\

{\bf Notation.} Constants are denoted by $c$ or $C$ and may change from line to line, even within the same computation. Subscripts indicate dependence on additional parameters; for instance, $C_a$ denotes a constant whose dependence on $a$ is being emphasized (it may still change from line to line).
We use the convention $\mathbb N=\{0,1,2,\dots\}$ for natural numbers. The open ball of radius $r>0$ centered at $x_0\in\R^n$ is denoted by $B_r(x_0)=\{x\in\R^n:\ |x-x_0|<r\}$, and we write $B_r=B_r(0)$.
For a multiindex $\nu=(\nu_1,\ldots,\nu_n)\in\mathbb N^n$, write $|\nu|=\sum_i\nu_i$, $\nu!=\prod_i\nu_i!$, $x^\nu=\prod_i x_i^{\nu_i}$ and $D^\nu=\partial_1^{\nu_1}\cdots\partial_n^{\nu_n}$. 
Unless otherwise specified, integrals are taken with respect to the $n$-dimensional Lebesgue measure, and the differential symbol $dx$ is omitted. We explicitly write $d\mu$ when integration is taken with respect to the weighted measure $\mu$.

\section{$C^{0,\alpha}$ estimates under Lipschitz coefficients}\label{sec:holder}
The aim of this section is the proof of Theorem \ref{thm:holder}.

Let us define $d\mu=u^2 \,dx$. Since $u$ is nontrivial in $B_2$, unique continuation gives $\mu(B_1)=\|u\|^2_{L^2(B_1)}>0$. Define $w=v/u$ on $\{u\neq0\}$ and assign arbitrary values on $Z(u)$. This defines an element of $L^2(B_2,d\mu)$ with $\|w\|_{L^2(B_2,d\mu)}=\|v\|_{L^2(B_2)}$.

\begin{Lemma}\label{lem:mass}
Assume \eqref{eq:doubling}. For every fixed $\rho\in(0,1/4]$, there exists a constant $c_\rho=c(n,D,\rho)>0$ such that
\begin{equation}\label{eq:mass-lower}
 \mu(B_\rho(x_0))\ge c_\rho\mu(B_1)\qquad \forall \, x_0\in B_{3/2}.
\end{equation}
\end{Lemma}
\begin{proof}
Put $r=1/4$. Cover $B_1$ by $J_n$ balls $B_r(z_j)$ whose centers lie in $B_1$. There exists at least one index $i\in\{1,\dots,J_n\}$ such that $\mu(B_r(z_i))\ge\mu(B_1)/J_n$, otherwise the sum of the measures of $B_r(z_j)$ would be smaller than $\mu(B_1)$, contradicting the fact that they cover $B_1$.

Fix $x_0\in B_{3/2}$ and connect $z_i$ to $x_0$ by a segment. Divide the segment into $N$ pieces with $N\le10$ of length at most $r=1/4$. The latter bounds on the number of pieces and their length are possible since $|x_0-z_i|<5/2$. Each piece has endpoints $p_j,p_{j+1}$ for any $j=0,\dots,N-1$, with $p_0=z_i$ and $p_N=x_0$. By convexity, any $p_j$ lies in $B_{3/2}$.

Notice that for any $j=0,\dots,N-1$ we have the inclusion $B_r(p_{j})\subseteq B_{2r}(p_{j+1})$, which implies, together with \eqref{eq:doubling}
$$
 \mu(B_r(p_{j}))\le \mu(B_{2r}(p_{j+1}))\le D\mu(B_r(p_{j+1})).
$$
Iterating,
$\mu(B_r(x_0))\ge D^{-10}\mu(B_1)/J_n$.

For $0<\rho\le r=1/4$, choose the least integer $k$ with $2^k\rho\ge r$, with the choice $k=0$ if $\rho=r$. Notice that, by minimality of $k$, we have $2^{k-1}\rho<r=1/4$. Notice that $k\to\infty$ as $\rho\to0$. Again, by inclusion and doubling \eqref{eq:doubling}, we have
$$
 \mu(B_r(x_0))\le\mu(B_{2^k\rho}(x_0))\le D^k\mu(B_\rho(x_0)).
$$
Take $c_\rho=(J_nD^{10+k})^{-1}$. This proves \eqref{eq:mass-lower}.
\end{proof}

\subsection{Best constant approximation}
For every center $x_0\in B_{3/2}$ and every radius $0<r\le1/4$, we find the best constant approximation of $w$ at that fixed scale. In other words, let us consider the minimization problem
\begin{equation}\label{eq:constant-approx}
\mathcal E_0(x_0,r)=\inf_{c\in\R}\left(\fint_{B_r(x_0)}|w-c|^2 \, d\mu\right)^{1/2}.
\end{equation}
If it exists, the minimizer is the orthogonal projection of $w$ onto the space of constant polynomials in $L^2(B_r(x_0),d\mu)$, while $\mathcal E_0(x_0,r)$ measures the corresponding projection error. We call $\mathcal E_0(x_0,r)$ the zeroth-order excess of $w$ at $x_0$ and scale $r$.

Here and throughout, the weighted average is understood as
\begin{equation*}
\fint_{B_r(x_0)} f \, d\mu=\frac{1}{\mu(B_r(x_0))}\int_{B_r(x_0)} f \, d\mu.
\end{equation*}

\begin{Lemma}\label{lem:constant-fit}
Let $x_0\in B_{3/2}$ and $0<r\le 1/4$. The following properties hold:
\begin{itemize}
\item[i)] the minimization problem \eqref{eq:constant-approx} admits a unique minimizer $c_{x_0,r}$, with the property
\begin{equation}\label{eq:ortho-constant}
\int_{B_r(x_0)}(w-c_{x_0,r})\,d\mu=0.
\end{equation}
Hence, the constant fit is
\begin{equation}\label{eq:best-constant}
 c_{x_0,r}=\fint_{B_r(x_0)}w \, d\mu.
\end{equation}

\item[ii)] for every $r\le s\le \min\{2r,1/4\}$, one has
\begin{equation}\label{eq:excess-control-constants}
|c_{x_0,r}-c_{x_0,s}|\le \mathcal E_0(x_0,r)+\sqrt{D}\,\mathcal E_0(x_0,s).
\end{equation}

\item[iii)] for every fixed $\rho\in(0,1/4]$, there exists $C=C(n,D,\rho)>0$ such that for every $x_0\in B_{3/2}$ and for every $\rho\le r\le1/4$
\begin{equation}\label{eq:uniform-bound-constant}
\mathcal E_0(x_0,r)+|c_{x_0,r}|\le C \frac{\|w\|_{L^2(B_2,d\mu)}}{\mu(B_1)^{1/2}}=C \frac{\|v\|_{L^2(B_2)}}{\|u\|_{L^2(B_1)}}.
\end{equation}
\end{itemize}
\end{Lemma}

\begin{proof}
The point $i)$ is equivalent to proving existence and uniqueness for
\begin{equation*}
\inf_{c\in\R}\|w-c\|_{L^2(B_r(x_0),d\mu)}.
\end{equation*}
The constants form a one-dimensional linear subspace $\mathcal P_0$ of $L^2(B_r(x_0),d\mu)$. Since $\mu(B_r(x_0))>0$, this subspace is nontrivial, and by finite dimensionality, it is closed. The Hilbert-space projection theorem gives a unique best constant $c_{x_0,r}$ minimizing the norm in $\mathcal P_0$. Moreover, the residual $w-c_{x_0,r}$ is orthogonal in $L^2(B_r(x_0),d\mu)$ to every constant function; testing against $1$ gives \eqref{eq:ortho-constant}, which in turn gives \eqref{eq:best-constant}.

In order to prove $ii)$, let us estimate the distance between two best constants at different scales
\begin{align*}
 |c_{x_0,r}-c_{x_0,s}|
 &=\left(\fint_{B_r(x_0)}|c_{x_0,r}-c_{x_0,s}|^2 \,d\mu\right)^{1/2}\\
 &\le \mathcal E_0(x_0,r)
   +\left(\frac{\mu(B_{s}(x_0))}{\mu(B_r(x_0))}\right)^{1/2}
       \mathcal E_0(x_0,s)\\
 &\le \mathcal E_0(x_0,r)+\sqrt D\,\mathcal E_0(x_0,s).
\end{align*}
In the previous inequalities, we added and subtracted $w$ in the integral, we used the Minkowski inequality, the assumption $r\leq s$, the definition of the excess, the assumption $s\leq 2r$, and the doubling inequality from scale $2r$ to $r$.

For the last point $iii)$, choosing the competitor $c=0$ in the minimization \eqref{eq:constant-approx}, we get
\begin{equation*}
\mathcal E_0(x_0,r)\le\frac{\|w\|_{L^2(B_2,d\mu)}}{\mu(B_r(x_0))^{1/2}}.
\end{equation*}
Moreover, by the Cauchy-Schwarz inequality in \eqref{eq:best-constant},
\begin{equation*}
 |c_{x_0,r}|\le\frac{\|w\|_{L^2(B_2,d\mu)}}{\mu(B_r(x_0))^{1/2}}.
\end{equation*}
Then, Lemma \ref{lem:mass} and $r\geq\rho$ give 
$$\mu(B_r(x_0))\geq\mu(B_\rho(x_0))\ge c_\rho\mu(B_1).$$
The latter gives \eqref{eq:uniform-bound-constant}.
\end{proof}

\begin{remark}\label{rem:constant}
In the next argument, it will be convenient to rewrite the best approximation constant \eqref{eq:best-constant} and the excess \eqref{eq:constant-approx} in terms of $u,v$; that is, 
\begin{equation}\label{eq:constant-fit}
 c_{x_0,r}=\frac{1}{\mu(B_r(x_0))}\int_{B_r(x_0)}uv,\qquad
 \mathcal E_0(x_0,r)=\left(\frac{1}{\mu(B_r(x_0))}\int_{B_r(x_0)}|v-c_{x_0,r}u|^2\right)^{1/2}.
\end{equation}
\end{remark}

\subsection{Decay of the zeroth-order excess}

From now on, let's define the normalized $L^2(d\mu)$-size of the quotient
\begin{equation}\label{eq:ratio-norms}
\mathcal N(u,v)=\frac{\|w\|_{L^2(B_2,d\mu)}}{\mu(B_1)^{1/2}}=\frac{\|v\|_{L^2(B_2)}}{\|u\|_{L^2(B_1)}}.
\end{equation}

The next result shows that the assumptions of Theorem \ref{thm:holder} imply a quantitative decay of the zeroth-order excess.

\begin{Proposition}\label{prop:constant-excess}
Fix $r_0=1/4$. The assumptions of Theorem \ref{thm:holder} imply that, for any fixed $\alpha\in(0,1)$, there exists $C_\alpha=C(n,\lambda,\Lambda,L,D,\alpha)>0$ such that for any $x_0\in B_1$ and any $0<r\le r_0$
\begin{equation}\label{eq:constant-excess-bound}
 \mathcal E_0(x_0,r)\le C_\alpha\mathcal N(u,v)r^\alpha.
\end{equation}
\end{Proposition}

\begin{proof}
Without loss of generality $v$ is nontrivial in $B_2$. Replace $u$ by $u/\|u\|_{L^2(B_1)}$ and $v$ by $v/\|v\|_{L^2(B_2)}$. The new quotient is $w/\mathcal N(u,v)$ and the new excess is $\mathcal E_0/\mathcal N(u,v)$. The equations, zeros, and the doubling
constant $D$ are unchanged. We prove the estimate for these normalized
pairs, for which $\mu(B_1)=\|w\|_{L^2(B_2,d\mu)}^2=\|v\|_{L^2(B_2)}^2=1$, using
the fixed radius $r_0=1/4$ throughout this proof.\\

{\bf Step 1: the contradiction argument and the choice of the scale.}
Let us suppose by contradiction that \eqref{eq:constant-excess-bound} does not hold among normalized solutions. Define
\begin{equation*}
\Theta(\rho)=\sup_{A,u,v}\sup_{x_0\in B_1}\sup_{\rho\leq r\leq r_0}\frac{\mathcal E_0(x_0,r)}{r^\alpha}.
\end{equation*}
The first supremum is taken over all normalized triples satisfying the fixed bounds $n,\lambda,\Lambda,L,D$. Lemma \ref{lem:constant-fit} point $iii)$ implies that for any fixed $\rho\in(0,r_0]$ the quantity $\Theta(\rho)$ is finite. So the contradiction implies that $\Theta(\rho)\to\infty$ as $\rho\to0$.

As $k\to\infty$, let us choose triples $(A_k,u_k,v_k)$ and sequences $\rho_k\to0$, $r_k\in[\rho_k,r_0]$, $x_k\in B_1$ such that
\begin{equation*}
\frac{\mathcal E_0(x_k,r_k)}{r_k^\alpha}\geq \frac{\Theta(\rho_k)}{2}.
\end{equation*}
This is possible by definition of the supremum. Set
$$
 \sigma_k=\mathcal E_0(x_k,r_k),\qquad M_k=\sigma_k/r_k^\alpha.
$$
Then $M_k\to\infty$. Moreover, for any $s\in[r_k,r_0]$ we have
\begin{equation}\label{eq:scale-s}
\mathcal E_0(x_k,s)\leq 2M_ks^\alpha.
\end{equation}
In fact $s\geq r_k\geq\rho_k$ and then $\mathcal E_0(x_k,s)/s^\alpha\leq \Theta(\rho_k)\leq 2M_k$. Notice that $r_k\to0$. In fact, if this is not the case we get for a subsequence $r_k\geq\overline\rho>0$, which implies
$$M_k=\frac{\mathcal E_0(x_k,r_k)}{r_k^\alpha}\leq\Theta(\overline\rho),$$
and the finiteness of $\Theta(\overline\rho)$ is in contradiction with $M_k\to\infty$.\\

{\bf Step 2: blow-up sequences and normalization.}
Let us denote by $c_k$ the best approximation constant at $x_k$ and scale $r_k$; that is, $c_k=c_{x_k,r_k}$. Let us also define
\begin{equation*}
\kappa_k^2=\int_{B_1}u_k^2(x_k+r_kx)=\frac{1}{r_k^n}\int_{B_{r_k}(x_k)}u_k^2.
\end{equation*}
Then, the rescaled coefficients are given by
\begin{equation*}
\tilde A_k(x)=A_k(x_k+r_kx),
\end{equation*}
and the blow-up sequences are given by
\begin{equation*}
\tilde u_k(x)=\frac{u_k(x_k+r_kx)}{\kappa_k},\qquad \tilde v_k(x)=\frac{(v_k-c_ku_k)(x_k+r_kx)}{\kappa_k\sigma_k}.
\end{equation*}
Notice that the equations, together with the zero set inclusion, are preserved; that is,
\begin{equation}\label{eq:zero-exact-pair}
 L_{\tilde A_k}\tilde u_k=L_{\tilde A_k}\tilde v_k=0,\qquad Z(\tilde u_k)\subseteq Z(\tilde v_k).
\end{equation}
These functions are defined on the blow-up domains
$$
 \Omega_k=\frac{B_2-x_k}{r_k}.
$$
Notice that $B_{1/r_k}\subset\Omega_k$, and since $r_k\to 0$, for any $R>0$ the ball $B_R\subseteq\Omega_k$ for $k$ big enough. Thus, $\Omega_k$ expands to $\R^n$.

The definition of $\sigma_k=\mathcal E_0(x_k,r_k)$ and $c_k=c_{x_k,r_k}$ in \eqref{eq:constant-approx} and \eqref{eq:best-constant} can be rewritten in terms of $u_k,v_k$ as in \eqref{eq:constant-fit}; that is,
\begin{equation*}
\mathcal E_0(x_k,r_k)=\left(\frac{\int_{B_1}|v_k-c_k u_k|^2(x_k+r_kx)}{\int_{B_1}u_k^2(x_k+r_kx)}\right)^{1/2},\qquad c_{x_k,r_k}=\frac{\int_{B_1}u_kv_k(x_k+r_kx)}{\int_{B_1}u_k^2(x_k+r_kx)},
\end{equation*}
so that
\begin{equation}\label{eq:normalization-uvk}
\int_{B_1}\tilde u_k^2=\int_{B_1}\tilde v_k^2=1,\qquad \int_{B_1}\tilde u_k\tilde v_k=0.
\end{equation}

{\bf Step 3: integral growth of the blow-up sequences.}
In this step we prove that the rescaled pair $\tilde u_k,\tilde v_k$ does not only have the normalization on $B_1$, but also has controlled growth on every fixed larger ball $B_R$ with $R>1$. This is essential for two later things: compactness on arbitrary $B_R$, and the Liouville argument for the blow-up limit.

For $R=1$, the bounds below follow from \eqref{eq:normalization-uvk}. Fix $R>1$. Since $r_k\to0$, one has $Rr_k\leq r_0$ for every $k\geq \overline k(R)$. Then, let's consider $J\in\mathbb N\setminus\{0\}$ such that $2^{J-1}<R\leq 2^J$, and for $i=0,\dots,J$, let's consider $c_{k,i}=c_{x_k,s_i}$, where $s_i=2^ir_k$ for any $i=0,\dots,J-1$ and $s_J=Rr_k$. By \eqref{eq:scale-s}, we have
\begin{equation}\label{eq:scale-s2}
\mathcal E_0(x_k,s_i)\leq 2M_ks_i^\alpha.
\end{equation}
By the previous inequality, and by Lemma \ref{lem:constant-fit} point $ii)$, since $s_i\le s_{i+1}\le \min\{2s_i,1/4\}$ for any $i=0,\dots,J-1$, one has
\begin{equation}\label{eq:excess-control-constants2}
|c_{k,i}-c_{k,i+1}|\le \mathcal E_0(x_k,s_i)+\sqrt{D}\,\mathcal E_0(x_k,s_{i+1})\leq C_D M_ks_{i}^\alpha.
\end{equation}
Then, summing up the pieces, we have
\begin{equation*}
|c_{k,0}-c_{k,J}|\leq C_{\alpha,D} M_k(Rr_k)^\alpha.
\end{equation*}
For the previous summation notice that, since $2^{J-1}<R$, the geometric-series estimates give
\begin{equation}\label{eq:geometricsum}
\sum_{i=0}^{J-1}s_i^{\alpha}=r_k^\alpha\sum_{i=0}^{J-1}(2^{\alpha})^i\leq(2^{J-1}r_k)^{\alpha}\sum_{i=0}^{\infty}(2^{-\alpha})^i\leq C_\alpha(Rr_k)^{\alpha}.
\end{equation}

Then, in $B_{Rr_k}(x_k)$, adding and subtracting $c_{k,J}u_k$, and by the definition of the excess at scale $Rr_k$, we get
\begin{eqnarray*}
\|v_k-c_{k,0}u_k\|_{L^2(B_{Rr_k}(x_k))}&\leq& \big(\mathcal E_0(x_k,Rr_k)+|c_{k,0}-c_{k,J}|\big)\|u_k\|_{L^2(B_{Rr_k}(x_k))}\\
&\leq& C_{\alpha,D} M_k(Rr_k)^\alpha\|u_k\|_{L^2(B_{Rr_k}(x_k))}.
\end{eqnarray*}
After changing variables and dividing by $\kappa_k\sigma_k=\kappa_kM_kr_k^\alpha$, we have
\begin{equation}\label{eq:growth-vkuk}
\|\tilde v_k\|_{L^2(B_{R})}\leq C_{\alpha,D} R^\alpha \|\tilde u_k\|_{L^2(B_{R})}.
\end{equation}

We finally control the growth of $\tilde u_k$. Let $N_0=\lceil\log_2R\rceil$,
so that $R\leq 2^{N_0}<2R$. By the definition of $\tilde u_k$,
\begin{align}\label{eq:rascaled-u-R}
\int_{B_R}\tilde u_k^2&=\frac{\displaystyle\int_{B_{Rr_k}(x_k)}u_k^2}{\displaystyle\int_{B_{r_k}(x_k)}u_k^2}.
\end{align}

Since
$B_{Rr_k}(x_k)\subset B_{2^{N_0}r_k}(x_k)$, we can iterate the doubling inequality $N_0$ times. Indeed, for every $i=0,\dots,N_0-1$,
$$
2^ir_k\leq2^{N_0-1}r_k<Rr_k\leq r_0\leq1/4,
$$
so the doubling inequality is applicable at every step. Therefore
$$
\int_{B_{Rr_k}(x_k)}u_k^2\leq\int_{B_{2^{N_0}r_k}(x_k)}u_k^2\leq D^{N_0}\int_{B_{r_k}(x_k)}u_k^2,
$$
and consequently, we obtain the integral growth estimate
\begin{equation}\label{eq:polynomialbound-tildeuk}
\int_{B_R}\tilde u_k^2\leq D^{\lceil\log_2R\rceil}\leq D R^{\log_2D}.
\end{equation}

{\bf Step 4: compactness and the limiting equations.}
Extract a subsequence for which $\tilde A_k(0)$ converges to some constant coefficient matrix $A_0$. On every fixed ball $B_R$, $[\tilde A_k]_{C^{0,1}(B_R)}\le Lr_k$ and
$\|\tilde A_k-\tilde A_{k}(0)\|_{L^\infty(B_R)}\le C_R Lr_k$. Thus $\tilde A_k\to A_0$ locally uniformly.

The growth bounds \eqref{eq:polynomialbound-tildeuk} and \eqref{eq:growth-vkuk} of $\tilde u_k,\tilde v_k$ yield local $L^2$-boundedness; interior elliptic estimates (see \eqref{eq:E1} in Remark \ref{rem:second-order-schauder}) then upgrade this to local $C^{1,\eta}$-compactness for any $\eta\in(0,1)$. After passing to a subsequence, let us define $U,V$, the blow-up limits of $\tilde u_k,\tilde v_k$; that is, $\tilde u_k\to U$ and $\tilde v_k\to V$ in $C^{1,\eta}(B_R)$ for any $\eta\in(0,1)$ and $R>0$. This defines the blow-up profiles in the whole of $\R^n$, by a compact exhaustion of $\R^n$, and a standard diagonal argument along subsequences. It is easy to check that $U,V$ are entire solutions to
\begin{equation*}
\mathrm{div}(A_0\nabla U)=\mathrm{div}(A_0\nabla V)=0\qquad\mathrm{in \ }\R^n.
\end{equation*}
The limits of the normalization \eqref{eq:normalization-uvk} are
\begin{equation}\label{eq:zero-limit-normalization}
 \int_{B_1}U^2=\int_{B_1}V^2=1,\qquad\int_{B_1}UV=0.
\end{equation}
In particular $U$ and $V$ are nontrivial. Moreover, the growth conditions \eqref{eq:polynomialbound-tildeuk} and \eqref{eq:growth-vkuk}, in the limit, imply that $U,V$ possess polynomial growth too. In particular
\begin{equation}\label{eq:polynomialbound-UV}
\|U\|_{L^2(B_R)}\leq CR^{\frac12\log_2D},\qquad \|V\|_{L^2(B_{R})}\leq C R^\alpha \|U\|_{L^2(B_{R})}.
\end{equation}

{\bf Step 5: transfer of the nodal inclusion.}
We now show that the nodal set inclusion is preserved in the blow-up limit, namely $Z(U)\subseteq Z(V)$. Let $x_0\in Z(U)$ and fix $\rho>0$. Since $U$ is a nontrivial $L_{A_0}$-harmonic function and $U(x_0)=0$, it must change sign in $B_\rho(x_0)$. Indeed, otherwise $U$ would have one sign in the ball and the strong maximum principle and the unique continuation principle would imply that $U$ vanishes identically there, hence everywhere. Therefore, there exist points $x_\rho^+,x_\rho^-\in B_\rho(x_0)$ such that $U(x_\rho^+)>0$ and $U(x_\rho^-)<0$. By the local uniform convergence $\tilde u_k\to U$, the same strict inequalities hold for $\tilde u_k$ at $x_\rho^\pm$ for all $k$ sufficiently large. Since $B_\rho(x_0)$ is convex, the segment joining $x_\rho^+$ and $x_\rho^-$ is contained in $B_\rho(x_0)$, and by continuity there exists a point $x^0_{k,\rho}\in [x_\rho^+,x_\rho^-]\subset B_\rho(x_0)$ such that $\tilde u_k(x^0_{k,\rho})=0$. The nodal set inclusion for the rescaled pair then gives $\tilde v_k(x^0_{k,\rho})=0$. To pass to the limit, choose $\rho_j\to0$ and $k_j\to\infty$ sufficiently fast so that the previous construction is valid with $\rho=\rho_j$ and $k=k_j$. Setting $x_j:=x^0_{k_j,\rho_j}$, since the balls are shrinking toward $x_0$,
we have $x_j\to x_0$. Since $\tilde v_{k_j}(x_j)=0$, $\tilde v_{k_j}\to V$ locally uniformly and $x_j\to x_0$, it follows that $V(x_0)=0$. Thus $x_0\in Z(V)$.\\

{\bf Step 6: contradiction by polynomial Liouville theorems and Murdoch's division lemma.}
Apply the change $x=Sy$, where $S=A_0^{1/2}$. The transformed functions
$\tilde U(y)=U(Sy)$, $\tilde V(y)=V(Sy)$ are harmonic in $\R^n$ and retain similar $L^2$-integral growth conditions, with the same exponents. In particular
\begin{equation}\label{eq:polynomialbound-tildeU}
\|\tilde U\|_{L^2(B_R)}\leq CR^{\frac12\log_2D}.
\end{equation}
Theorem \ref{thm:app-liouville}, with $m=0$, therefore implies that $\tilde U$ is a harmonic polynomial $Q$. Set $N=\deg Q$. Here the growth exponent in Theorem \ref{thm:app-liouville} is $\gamma=(\log_2D-n)/2\ge0$ (recall that $D\geq 2^n$). In particular, $N\le\lfloor\gamma\rfloor$.

Set $c_+=\max\{1,\sqrt\Lambda\}$, $c_-=\min\{1,\sqrt\lambda\}$ and $\tilde c=c_+/c_-$.
The inclusions $SB_R\subset B_{c_+R}$ and $S^{-1}B_{c_+R}\subset B_{\tilde cR}$, together with change of variables, give for any $R\geq1$
\begin{equation}\label{eq:changed-relative-growth}
 \|\tilde V\|_{L^2(B_R)}\le C R^\alpha\|\tilde U\|_{L^2(B_{\tilde cR})}.
\end{equation}
Here $C$ depends only on the original growth constant, $\alpha,\lambda,\Lambda$.
Since $\tilde U=Q$ is a polynomial of degree $N$, its last norm is bounded by $CR^{N+n/2}$, with $C$ independent of $R$. Then,
\begin{equation*}
\|\tilde V\|_{L^2(B_{R})}\leq CR^{\alpha+N+\frac{n}{2}}.
\end{equation*}
Theorem \ref{thm:app-liouville} with $m=0$ implies that also $\tilde V$ is a harmonic polynomial $P$ of degree $\deg P\leq\lfloor \alpha+N\rfloor$. Since $\alpha\in(0,1)$, then $\lfloor \alpha+N\rfloor=N$. Hence $\deg P\leq\deg Q$.

Of course, the applied linear change of coordinates does preserve the zero set inclusion; that is, $Z(Q)\subseteq Z(P)$. Since $Q$ is a nontrivial harmonic polynomial, Murdoch's division Lemma \ref{lem:murdoch} yields $P=QR$ for some polynomial $R$. The degree inequality $\deg P\leq\deg Q$ implies that actually $\deg P=\deg Q$ and $R$ is constant. Hence $P=cQ$ for some $c\in\mathbb R$.

Returning to the original coordinates yields $V=cU$ in $\R^n$, since composition with an invertible linear map preserves both polynomiality and polynomial degree. The normalizations \eqref{eq:zero-limit-normalization} imply
$$
 0=\int_{B_1}UV=c\int_{B_1}U^2=c.
$$
Hence $V=0$, contradicting $\int_{B_1}V^2=1$. The contradiction bounds $\Theta(\rho)$ uniformly as $\rho\to0$. Restoring the independent normalizations of $u$ and $v$ proves \eqref{eq:constant-excess-bound}.
\end{proof}

\subsection{Proof of Theorem \ref{thm:holder}}
We finally complete the proof of the $C^{0,\alpha}$ estimate.
\begin{proof}[Proof of Theorem \ref{thm:holder}]
Set $r_0=1/4$ and $K=C_\alpha\mathcal N(u,v)$, so that Proposition \ref{prop:constant-excess} gives
$$
\mathcal E_0(x_0,r)\leq Kr^\alpha\qquad\forall\,x_0\in B_1,\quad 0<r\leq r_0.
$$
Let
$$
r_j=2^{-j}r_0,\qquad j\in\mathbb N.
$$
By Lemma \ref{lem:constant-fit} point $ii)$,
$$
|c_{x_0,r_j}-c_{x_0,r_{j+1}}|\leq CKr_j^\alpha.
$$
Therefore $\{c_{x_0,r_j}\}_{j\in\mathbb N}$ is a Cauchy sequence and we may define
$$
\overline w(x_0)=\lim_{j\to\infty}c_{x_0,r_j}.
$$
Summing the geometric series gives
\begin{equation}\label{eq:constant-to-limit}
|\overline w(x_0)-c_{x_0,r_j}|\leq CK\sum_{l=j}^\infty r_l^\alpha=CKr_j^\alpha\sum_{m=0}^\infty (2^{-\alpha})^m\leq CKr_j^\alpha.
\end{equation}

Moreover, by the fixed-scale estimate \eqref{eq:uniform-bound-constant} in Lemma \ref{lem:constant-fit} point $iii)$,
$$
|c_{x_0,r_0}|\leq CK.
$$
Hence
$$
|\overline w(x_0)|\leq |c_{x_0,r_0}|+\sum_{j=0}^\infty |c_{x_0,r_j}-c_{x_0,r_{j+1}}|\leq CK,
$$
and therefore
\begin{equation}\label{eq:holder-Linfty}
\|\overline w\|_{L^\infty(B_1)}\leq CK.
\end{equation}

We now estimate the H\"older seminorm. Let $x_0,y_0\in B_1$ and set
$h=|x_0-y_0|$. Assume first that $0<h\leq r_0/2$, and choose $j\geq1$
such that $r_{j+1}<h\leq r_j$. Then $B_{r_j}(y_0)\subset B_{2r_j}(x_0)=B_{r_{j-1}}(x_0)$. Hence

\begin{align*}
 |c_{y_0,r_j}-c_{x_0,r_{j-1}}|
 &=\left(\fint_{B_{r_j}(y_0)}|c_{y_0,r_j}-c_{x_0,r_{j-1}}|^2 \,d\mu\right)^{1/2}\\
 &\le \mathcal E_0(y_0,r_j)
   +\left(\frac{\mu(B_{r_{j-1}}(x_0))}{\mu(B_{r_j}(y_0))}\right)^{1/2}
       \mathcal E_0(x_0,r_{j-1}).
\end{align*}
By the excess estimate and the doubling inequality, since $B_{r_{j-1}}(x_0)=B_{2r_j}(x_0)\subset B_{4r_j}(y_0)$, we have
$$
|c_{y_0,r_j}-c_{x_0,r_{j-1}}|\leq CKr_j^\alpha.
$$
Using \eqref{eq:constant-to-limit} at $x_0$ and $y_0$ in $B_1$, we obtain
$$
|\overline w(x_0)-\overline w(y_0)|\leq CKr_j^\alpha\leq CK h^\alpha,
$$
because $r_j<2h$.

If $h>r_0/2$, then by \eqref{eq:holder-Linfty},
$$
|\overline w(x_0)-\overline w(y_0)|\leq 2\|\overline w\|_{L^\infty(B_1)}\leq CKh^\alpha,
$$
since $r_0$ is fixed.

Finally, if $u(x_0)\neq0$, then $w=v/u$ is continuous in a neighborhood of
$x_0$, and therefore
$$
c_{x_0,r_j}=\fint_{B_{r_j}(x_0)}w\,d\mu\longrightarrow w(x_0).
$$
Thus $\overline w=w$ on $\{u\neq0\}$. Moreover, $\{u\neq0\}$ is dense
in $B_1$.
Hence $\overline w$ is the unique continuous extension of $w$ to $B_1$.
Together with the $L^\infty$ and H\"older estimates above, this proves
\eqref{eq:holder-conclusion}.
\end{proof}

Notice that, by a fixed rescaling and localization, the H\"older estimate can be extended to a slightly larger interior ball; that is, of radius $1<r^*<3/2$. For later purposes, we choose $r^*=5/4$.

\begin{remark}\label{rem:holder-scaled}
Under the assumptions of Theorem \ref{thm:holder}, for every $0<\alpha<1$ the quotient has a unique continuous extension to $B_{5/4}$ and
\begin{equation}\label{eq:holder-scaled}
 \|w\|_{C^{0,\alpha}(B_{5/4})}\le C(n,\lambda,\Lambda,L,D,\alpha)\mathcal N(u,v).
\end{equation}
\end{remark}

\section{$C^{1,\alpha}$ estimates under Lipschitz coefficients}\label{sec:first-order}
The aim of this section is the proof of Theorem \ref{thm:lip}. Throughout this section, we use the previously proved $C^{0,\beta}$ estimate in $B_{5/4}$, valid for every $\beta\in(0,1)$, as given in Remark \ref{rem:holder-scaled}. Throughout Section \ref{sec:first-order}, fix an auxiliary exponent $\eta\in(0,1)$ for the replacement estimates, for instance $\eta=1/2$. Unless otherwise specified, constants depend only on $n,\lambda,\Lambda,L,D$ and the indicated degrees and exponents.

\subsection{Best affine approximation}\label{sec:affine-fits}
For $x_0\in B_{3/2}$ and $0<r\le1/4$, define the $1$st-order excess
\begin{equation}\label{eq:affine-excess}
 \mathcal E_1(x_0,r)=\inf_{\ell\in\mathcal P_1}\left(\fint_{B_r(x_0)}|w-\ell|^2 \,d\mu\right)^{1/2}=\inf_{\ell\in\mathcal P_1}\left(\frac{1}{\mu(B_r(x_0))}\int_{B_r(x_0)}|v-\ell u|^2\right)^{1/2},
\end{equation}
where $\mathcal P_1$ is the space of polynomials of degree at most $1$.

The next lemma provides existence and uniqueness of the weighted affine fit.

\begin{Lemma}\label{lem:affine-minimizer}
For every center $x_0\in B_{3/2}$ and every radius $0<r\le1/4$, the minimization problem \eqref{eq:affine-excess} has a unique minimizer. We denote it by
$$
 \ell_{x_0,r}(x)=a_{x_0,r}+b_{x_0,r}\cdot(x-x_0).
$$
Moreover, it is characterized by the normal equations
\begin{equation}\label{eq:affine-normal-equations}
 \int_{B_r(x_0)}(w-\ell_{x_0,r})\ell \,d\mu=0\qquad\text{for every }\ell\in\mathcal P_1.
\end{equation}
\end{Lemma}
\begin{proof}
The minimization \eqref{eq:affine-excess} is equivalent to
\begin{equation*}
\inf_{\ell\in\mathcal P_1}\|w-\ell\|_{L^2(B_r(x_0),d\mu)}.
\end{equation*}
Distinct affine polynomials define distinct elements of $L^2(B_r(x_0),d\mu)$: if $\int_{B_r(x_0)}\ell^2\,d\mu=0$, continuity of $u$ and the positive $\mu$-measure of the ball (Lemma \ref{lem:mass}) give a smaller open ball where $u\ne0$, on which $\ell$ vanishes. Thus $\ell$ is identically zero. Consequently, the space $\mathcal P_1$ is an $(n+1)$-dimensional, hence closed, linear subspace of $L^2(B_r(x_0),d\mu)$. The Hilbert-space projection theorem gives a unique best affine function $\ell_{x_0,r}$ minimizing the norm. Moreover, the residual $w-\ell_{x_0,r}$ is orthogonal in $L^2(B_r(x_0),d\mu)$ to every element $\ell\in\mathcal P_1$.
\end{proof}

Then, we prove the uniform weighted norm equivalence for affine functions.
\begin{Lemma}\label{lem:affine-moments}
There is $c_0=c_0(n,\lambda,\Lambda,L,D)>0$ such that,
for every $x_0\in B_{3/2}$ and $0<r\le1/4$, and every $a\in\R$, $b\in\R^n$,
\begin{equation}\label{eq:affine-moments}
 c_0(|a|^2+r^2|b|^2)\le\fint_{B_r(x_0)}|a+b\cdot(x-x_0)|^2\,d\mu \le2(|a|^2+r^2|b|^2).
\end{equation}
\end{Lemma}
\begin{proof}
The upper bound follows immediately from
$$
|a+b\cdot(x-x_0)|^2\leq 2|a|^2+2r^2|b|^2\qquad\text{in }B_r(x_0).
$$
For the lower bound, let us define
\begin{equation*}
\kappa^2=\frac{1}{r^n}\int_{B_r(x_0)}u^2,\qquad u_{x_0,r}(y)=\frac{u(x_0+ry)}{\kappa},\qquad A_{x_0,r}(y)=A(x_0+ry).
\end{equation*}
The functions $u_{x_0,r}$ and $A_{x_0,r}$ are defined in $B_2$.
Moreover, $A_{x_0,r}$ has the same ellipticity constants as $A$ and $[A_{x_0,r}]_{C^{0,1}(B_2)}\leq rL\leq L$. Then, by a change of variables and the doubling inequality \eqref{eq:doubling}, 
\begin{equation}\label{eq:boundsB1B2}
\int_{B_1}u_{x_0,r}^2=1,\qquad \int_{B_2}u_{x_0,r}^2\leq D.
\end{equation}
After the change of variables $x=x_0+ry$, the desired lower bound in \eqref{eq:affine-moments} is equivalent to
$$
\int_{B_1}u_{x_0,r}^2(y)|a+rb\cdot y|^2\geq c_0(|a|^2+r^2|b|^2).
$$
If $|a|^2+r^2|b|^2=0$, there is nothing to prove. Otherwise, dividing
the coefficients $a$ and $rb$ by $\sqrt{|a|^2+r^2|b|^2}$, and denoting the normalized coefficients again by $a$ and $b$, it is enough to prove
$$
\int_{B_1}u_{x_0,r}^2(y)|a+b\cdot y|^2\geq c_0,
$$
among coefficients with $|a|^2+|b|^2=1$. Let us suppose by contradiction that such a bound does not hold true. Then, there exist normalized solutions $u_k$, and coefficients $a_k,b_k$ with $|a_k|^2+|b_k|^2=1$ such that
$$
 \int_{B_1}u_k^2=1,\quad \int_{B_2}u_k^2\le D,\quad
 \int_{B_1}u_k^2|a_k+b_k\cdot y|^2\longrightarrow0.
$$
By \eqref{eq:E1}, with any fixed exponent $\eta\in(0,1)$, the sequence
$u_k$ is uniformly bounded in $C^{1,\eta}(B_{3/2})$. Hence, after
passing to a subsequence, $u_k$ converges uniformly in $\overline{B_1}$ to a continuous $u_\infty$ with $\int_{B_1}u_\infty^2=1$.
Also $(a_k,b_k)\to(a_\infty,b_\infty)$ with $|a_\infty|^2+|b_\infty|^2=1$. Passing to the limit, by continuity and nonnegativity of the integrand, it gives $u_\infty(a_\infty+b_\infty\cdot y)=0$ in $B_1$. The limiting denominator is nonzero on some open ball since $\int_{B_1}u_\infty^2=1$, so that the affine polynomial vanishes on that ball and hence is identically zero. This contradicts the coefficient normalization and proves the uniform lower bound.
\end{proof}

In the next result, we provide quantitative control of the coefficients of the best affine approximation, together with control of their variation across scales.

\begin{Lemma}\label{lem:affine-comparison}
The following estimates hold:
\begin{itemize}
\item[$i)$] for every $x_0\in B_1$, $0<r\le1/4$, and $\beta\in(0,1)$,
\begin{equation}\label{eq:fit-size}
 |a_{x_0,r}-w(x_0)|+r|b_{x_0,r}|\le C_\beta\,\mathcal N(u,v)\,r^\beta,
\end{equation}
where $C_\beta=C(n,\lambda,\Lambda,L,D,\beta)>0$;

\item[$ii)$] for every $x_0\in B_{3/2}$, $0<r\le1/4$, and $r\le s\le\min\{2r,1/4\}$, one has
\begin{equation}\label{eq:fit-difference}
 |a_{x_0,r}-a_{x_0,s}|
 +r|b_{x_0,r}-b_{x_0,s}|\le C\bigl(\mathcal E_1(x_0,r)+\mathcal E_1(x_0,s)\bigr),
 \end{equation}
where $C=C(n,\lambda,\Lambda,L,D)>0$;
 
 \item[$iii)$] for every fixed $\rho\in(0,1/4]$, there exists $C=C(n,D,\rho)>0$ such that for every $x_0\in B_{3/2}$ and for every $\rho\le r\le1/4$
 \begin{equation}\label{eq:uniform-bound-affine}
 \mathcal E_1(x_0,r)\leq C \mathcal N(u,v).
 \end{equation}
\end{itemize}
\end{Lemma}
\begin{proof}
Let us prove point $i)$. By Lemma \ref{lem:affine-minimizer}, $\ell_{x_0,r}$ is the orthogonal
projection of $w$ onto the space of affine polynomials $\mathcal P_1$ in $L^2(B_r(x_0),d\mu)$. Since constants belong to this space and the orthogonal projection is linear, the projection of $w-w(x_0)$ is $\ell_{x_0,r}-w(x_0)$. Hence, by the contraction property of orthogonal projections, and the $\beta$-H\"older bound of $w$ for any $\beta\in(0,1)$ in \eqref{eq:holder-scaled} (notice that $B_r(x_0)\subset B_{5/4}$),
$$
\fint_{B_r(x_0)}|\ell_{x_0,r}-w(x_0)|^2\,d\mu\le\fint_{B_r(x_0)}|w-w(x_0)|^2\,d\mu\le C\mathcal N(u,v)^2r^{2\beta}.
$$
Applying the lower bound in \eqref{eq:affine-moments} of Lemma \ref{lem:affine-moments} to the affine polynomial $\ell_{x_0,r}-w(x_0)=a_{x_0,r}-w(x_0)+b_{x_0,r}\cdot(x-x_0)$
gives
$$
c_0(|a_{x_0,r}-w(x_0)|^2+r^2|b_{x_0,r}|^2)\le C \mathcal N(u,v)^2r^{2\beta},
$$
and therefore \eqref{eq:fit-size}.

For point $ii)$, we want to estimate the distance between two affine fits at different scales

\begin{align*}
 \sqrt{c_0(|a_{x_0,r}-a_{x_0,s}|^2+r^2|b_{x_0,r}-b_{x_0,s}|^2)}
 &\le\left(\fint_{B_r(x_0)}|\ell_{x_0,r}-\ell_{x_0,s}|^2 \,d\mu\right)^{1/2}\\
 &\le \mathcal E_1(x_0,r)
   +\left(\frac{\mu(B_{s}(x_0))}{\mu(B_r(x_0))}\right)^{1/2}
       \mathcal E_1(x_0,s)\\
 &\le \mathcal E_1(x_0,r)+\sqrt D\,\mathcal E_1(x_0,s).
\end{align*}
In the previous inequalities, we used the lower bound in \eqref{eq:affine-moments}, we added and subtracted $w$ in the integral, we used the Minkowski inequality, the assumption $r\leq s$ and the definition of the excess. Then, by the assumption $s\leq 2r$, and the doubling inequality from scale $2r$ to $r$, \eqref{eq:fit-difference} follows.

For the last point $iii)$, choosing the competitor $\ell=0$ in the minimization \eqref{eq:affine-excess}, we get
\begin{equation*}
\mathcal E_1(x_0,r)\le\frac{\|w\|_{L^2(B_2,d\mu)}}{\mu(B_r(x_0))^{1/2}}.
\end{equation*}
Then, Lemma \ref{lem:mass} and $r\geq\rho$ give 
$$\mu(B_r(x_0))\geq\mu(B_\rho(x_0))\ge c_\rho\mu(B_1).$$
The latter gives \eqref{eq:uniform-bound-affine}.
\end{proof}

\subsection{Decay of the first-order excess}\label{sec:biharmonic}
The next result shows that the assumptions of Theorem \ref{thm:lip} imply a quantitative decay of the first-order excess.

\begin{Proposition}\label{prop:affine-excess}
Fix $r_0=1/4$. The assumptions of Theorem \ref{thm:lip} imply that, for every $\alpha\in(0,1)$, there is $C_\alpha=C(n,\lambda,\Lambda,L,D,\alpha)>0$ such that for any $x_0\in B_1$ and any $0<r\le r_0$
\begin{equation}\label{eq:affine-excess-bound}
 \mathcal E_1(x_0,r)\le C_\alpha\mathcal N(u,v)r^{1+\alpha}.
\end{equation}
\end{Proposition}

\begin{proof}
Without loss of generality $v$ is nontrivial in $B_2$. Replace $u$ by $u/\|u\|_{L^2(B_1)}$ and $v$ by $v/\|v\|_{L^2(B_2)}$. The new quotient is $w/\mathcal N(u,v)$ and the new excess is $\mathcal E_1/\mathcal N(u,v)$. The equations, zeros, and the doubling
constant $D$ are unchanged. We prove the estimate for these normalized
pairs, for which $\mu(B_1)=\|w\|_{L^2(B_2,d\mu)}^2=\|v\|_{L^2(B_2)}^2=1$, using
the fixed radius $r_0=1/4$ throughout this proof. Fix $\beta=(1+\alpha)/2\in(\alpha,1)$. Remark \ref{rem:holder-scaled} gives
\begin{equation}\label{eq:first-initial-holder}
 \|w\|_{C^{0,\beta}(B_{5/4})}\le C_\beta,
 \qquad C_\beta=C(n,\lambda,\Lambda,L,D,\beta).
\end{equation}
Every ball with center in $B_1$ and radius at most $r_0=1/4$ lies in $B_{5/4}$.\\

{\bf Step 1: the contradiction argument and the choice of the scale.}
Let us suppose by contradiction that \eqref{eq:affine-excess-bound} does not hold among normalized solutions. Define
\begin{equation*}
\Theta(\rho)=\sup_{A,u,v}\sup_{x_0\in B_1}\sup_{\rho\leq r\leq r_0}\frac{\mathcal E_1(x_0,r)}{r^{1+\alpha}}.
\end{equation*}
The first supremum is taken over all normalized triples satisfying the fixed bounds $n,\lambda,\Lambda,L,D$. Lemma \ref{lem:affine-comparison} point $iii)$ implies that for any fixed $\rho\in(0,r_0]$ the quantity $\Theta(\rho)$ is finite. So the contradiction implies that $\Theta(\rho)\to\infty$ as $\rho\to0$.

As $k\to\infty$, let us choose triples $(A_k,u_k,v_k)$ and sequences $\rho_k\to0$, $r_k\in[\rho_k,r_0]$, $x_k\in B_1$ such that
\begin{equation*}
\frac{\mathcal E_1(x_k,r_k)}{r_k^{1+\alpha}}\geq \frac{\Theta(\rho_k)}{2}.
\end{equation*}
This is possible by definition of the supremum. Set
$$
 \sigma_k=\mathcal E_1(x_k,r_k),\qquad M_k=\sigma_k/r_k^{1+\alpha}.
$$
Then $M_k\to\infty$. Moreover, for any $s\in[r_k,r_0]$ we have
\begin{equation}\label{eq:scale-ss}
\mathcal E_1(x_k,s)\leq 2M_ks^{1+\alpha}.
\end{equation}
In fact $s\geq r_k\geq\rho_k$ and then $\mathcal E_1(x_k,s)/s^{1+\alpha}\leq \Theta(\rho_k)\leq 2M_k$. Notice that $r_k\to0$. In fact, if this is not the case we get for a subsequence $r_k\geq\overline\rho>0$, which implies
$$M_k=\frac{\mathcal E_1(x_k,r_k)}{r_k^{1+\alpha}}\leq\Theta(\overline\rho),$$
and the finiteness of $\Theta(\overline\rho)$ is in contradiction with $M_k\to\infty$.\\

{\bf Step 2: blow-up sequences and normalization.}
Let us denote by $\ell_k$ the best affine approximation at $x_k$ and scale $r_k$; that is, $\ell_k(x)=\ell_{x_k,r_k}(x)=a_{x_k,r_k}+b_{x_k,r_k}\cdot(x-x_k)$. Then, rescaled coefficients and blow-up sequences are given by

\begin{align}
 \kappa_k^2&=\int_{B_1}u_k^2(x_k+r_kx),\qquad \tilde A_k(x)=A_k(x_k+r_kx),\notag\\
 \tilde u_k(x)&=\frac{u_k(x_k+r_kx)}{\kappa_k},\qquad\tilde v_k(x)=\frac{(v_k-\ell_{k}u_k)(x_k+r_kx)}{\kappa_k\sigma_k}.
 \label{eq:first-normalization}
\end{align}
These functions are defined on $\Omega_k=(B_2-x_k)/r_k$, which contains $B_{1/r_k}$. Thus every fixed ball $B_R$ with $R>0$ is contained in $\Omega_k$ for all sufficiently large $k$. Thus, the blow-up domain is $\R^n$. The equation for the sequence $\tilde u_k$ is preserved, as well as the nodal set inclusion; that is,
$$
 L_{\tilde A_k}\tilde u_k=0,\qquad  Z(\tilde u_k)\subseteq Z(\tilde v_k).
$$
As we will see, $\tilde v_k$ need not be $L_{\tilde A_k}$-harmonic.
The definition of the excess \eqref{eq:affine-excess} and the normal equations \eqref{eq:affine-normal-equations} for the affine fit give
\begin{equation}\label{eq:first-normal-equations}
 \int_{B_1}\tilde u_k^2=\int_{B_1}\tilde v_k^2=1,\qquad\int_{B_1}\tilde v_k\tilde u_k\ell=0\quad\forall\,\ell\in\mathcal P_1.
\end{equation}

{\bf Step 3: integral growth of the blow-up sequences.}
For $R=1$, the bounds below follow from \eqref{eq:first-normal-equations}. Fix $R>1$. Since $r_k\to0$, one has $Rr_k\leq r_0$ for every $k\geq \overline k(R)$. Then, let's consider $J\in\mathbb N\setminus\{0\}$ such that $2^{J-1}<R\leq 2^J$, and define $s_i=2^ir_k$ for any $i=0,\dots,J-1$ and $s_J=Rr_k$.

Write
$\ell_{k,i}(y)=a_{k,i}+b_{k,i}\cdot(y-x_k)$ for the best affine fit on $B_{s_i}(x_k)$; that is, $\ell_{k,i}=\ell_{x_k,s_i}$ with $a_{k,i}=a_{x_k,s_i}$ and $b_{k,i}=b_{x_k,s_i}$. In particular,
$\ell_{k,0}=\ell_{x_k,r_k}$.

By \eqref{eq:scale-ss}, we have
\begin{equation}\label{eq:scale-ss2}
\mathcal E_1(x_k,s_i)\leq 2M_ks_i^{1+\alpha}.
\end{equation}

Since $s_i\le s_{i+1}\le\min\{2s_i,r_0\}$, Lemma \ref{lem:affine-comparison} point $ii)$ and \eqref{eq:scale-ss2} imply for any $i=0,\dots,J-1$
$$
 |a_{k,i}-a_{k,i+1}|+s_i|b_{k,i}-b_{k,i+1}|\le C M_k s_i^{1+\alpha}.
$$
For $y\in B_{Rr_k}(x_k)$ and $i=0,\dots,J-1$ we have
$$
|\ell_{k,i}(y)-\ell_{k,i+1}(y)|\leq  |a_{k,i}-a_{k,i+1}|+|b_{k,i}-b_{k,i+1}|\cdot|y-x_k|\leq CM_k (s_i^{1+\alpha}+Rr_ks_i^\alpha).
$$
Hence, recalling that $2^{J-1}<R$, summation as in \eqref{eq:geometricsum} gives
\begin{equation*}
 |\ell_{k,0}(y)-\ell_{k,J}(y)|\le C M_k\sum_{i=0}^{J-1}\bigl(s_i^{1+\alpha}+Rr_ks_i^\alpha\bigr)\le C_\alpha M_k(Rr_k)^{1+\alpha}.
\end{equation*}

Then, in $B_{Rr_k}(x_k)$, adding and subtracting $\ell_{k,J}u_k$, and by the definition of the excess at scale $Rr_k$, we get
\begin{eqnarray*}
\|v_k-\ell_{k,0}u_k\|_{L^2(B_{Rr_k}(x_k))}&\leq& \big(\mathcal E_1(x_k,Rr_k)+\sup_{B_{Rr_k}(x_k)}|\ell_{k,0}(y)-\ell_{k,J}(y)|\big)\|u_k\|_{L^2(B_{Rr_k}(x_k))}\\
&\leq& C_{\alpha} M_k(Rr_k)^{1+\alpha}\|u_k\|_{L^2(B_{Rr_k}(x_k))}.
\end{eqnarray*}
After changing variables and dividing by $\kappa_k\sigma_k=\kappa_kM_kr_k^{1+\alpha}$, we have
\begin{equation}\label{eq:first-growth}
\|\tilde v_k\|_{L^2(B_R)}\le C_\alpha R^{1+\alpha}\|\tilde u_k\|_{L^2(B_R)}.
\end{equation}
The constant depends only on the fixed data and $\alpha$, and is independent of $R,k$.

As in the proof of Proposition \ref{prop:constant-excess}, iterating the doubling inequality $N_0=\lceil\log_2R\rceil$ times gives
\begin{equation}\label{eq:denominator-growth}
 \int_{B_R}\tilde u_k^2\le D^{\lceil\log_2R\rceil}\le D R^{\log_2D}.
\end{equation}
Consequently $\tilde u_k,\tilde v_k$ are bounded in $L^2$ on every fixed ball.\\

{\bf Step 4: compactness and the limiting equations.}
Compared with {\bf Step 4} in the proof of Proposition \ref{prop:constant-excess}, an additional argument is needed here, since the functions $\tilde v_k$ do not solve the same equation as $\tilde u_k$. 

Fix once and for all an exponent $\eta\in(0,1)$, for instance $\eta=1/2$. Consider a subsequence with $\tilde A_k(0)\to A_0$. For every fixed $R\ge1$,
\begin{equation}\label{eq:first-frozen-coefficients}
 \|\tilde A_k-\tilde A_k(0)\|_{C^{0,\eta}(B_{2R})}\le C_Rr_k.
\end{equation}
Indeed, using the Lipschitz continuity of $A_k$, the supremum is at most $2LRr_k$ and the $\eta$-H\"older seminorm is at most $Lr_k(4R)^{1-\eta}$.

Let us consider the normalized linear part of the affine fit $\ell_k$; that is,
$$
b_k=\frac{r_k}{\sigma_k}b_{x_k,r_k},
$$
and set
$$
v_k^*=\tilde v_k+(b_k\cdot x)\tilde u_k=\frac{(v_k-a_{x_k,r_k}u_k)(x_k+r_kx)}{\kappa_k\sigma_k}.
$$
Notice that $L_{\tilde A_k}v_k^*=0$. By \eqref{eq:fit-size}, $r_k|b_{x_k,r_k}|\le C_\beta r_k^{\beta}$, and hence, since $\sigma_k=M_kr_k^{1+\alpha}$ and $\beta>\alpha$,
\begin{equation}\label{eq:first-smallness}
 r_k|b_k|\le C_\beta\frac{r_k^{\beta-\alpha}}{M_k}
       \longrightarrow0.
\end{equation}
Notice that the above vanishing does not imply that $b_k$ is bounded.

Fix $R\ge1$ and take $k$ large enough that $B_{4R}\subset\Omega_k$. The growth estimates in {\bf Step 3} and the identity for $v_k^*$ give
$$
 \|\tilde u_k\|_{L^2(B_{4R})}\le C_R,\qquad \|v_k^*\|_{L^2(B_{4R})}\le C_R(1+|b_k|).
$$
The interior estimate \eqref{eq:E1} therefore yields
\begin{equation}\label{eq:first-exact-Schauder}
\|\tilde u_k\|_{C^{1,\eta}(B_{2R})}\le C_R,\qquad \|v_k^*\|_{C^{1,\eta}(B_{2R})}\le C_R(1+|b_k|).
\end{equation}
For simplicity, we write $L_0\cdot=\mathrm{div}(\tilde A_k(0)\nabla\cdot)$. Let us now consider the $L_0$-harmonic replacements $h_k^u,h_k^v$ of $\tilde u_k,v_k^*$ in $B_{2R}$; that is, the unique solutions of
$$
 \begin{cases}
 L_{0}h_k^u=0&\text{in }B_{2R},\\
 h_k^u=\tilde u_k&\text{on }\partial B_{2R},
 \end{cases}
 \qquad
 \begin{cases}
 L_{0}h_k^v=0&\text{in }B_{2R},\\
 h_k^v=v_k^*&\text{on }\partial B_{2R}.
 \end{cases}
$$
Equivalently, the errors $e_k^u=\tilde u_k-h_k^u$ and $e_k^v=v_k^*-h_k^v$ are the unique weak solutions in $H^1_0(B_{2R})$ of
$$
 -L_{0}e_k^u=\div\bigl((\tilde A_k-\tilde A_k(0))\nabla\tilde u_k\bigr),\qquad
 -L_{0}e_k^v=\div\bigl((\tilde A_k-\tilde A_k(0))\nabla v_k^*\bigr) \quad\text{in }B_{2R}.
$$
Existence and uniqueness follow from the Lax-Milgram theorem. Notice that $h_k^u,h_k^v,e_k^u,e_k^v$ depend on the choice of the radius $R$. Applying \eqref{eq:E2}, and taking into account \eqref{eq:first-frozen-coefficients} and \eqref{eq:first-exact-Schauder}, we have
\begin{equation}\label{eq:first-replacement-errors}
\|e_k^u\|_{C^{1,\eta}(B_{2R})}\le C_Rr_k,\qquad \|e_k^v\|_{C^{1,\eta}(B_{2R})}\le C_Rr_k(1+|b_k|).
\end{equation}
Define
$$
 G_{k,R}=h_k^v-(b_k\cdot x)h_k^u.
$$
Both $h_k^u,h_k^v$ are $L_0$-harmonic, and $b_k\cdot x$ is linear. Since $\tilde A_k(0)$ is constant, we have in $B_{2R}$
\begin{align}
 L_{0}G_{k,R}
   &=-2(\tilde A_k(0)b_k)\cdot\nabla h_k^u,\notag\\
 L_{0}^2G_{k,R}
   &=-2(\tilde A_k(0)b_k)\cdot\nabla(L_{0}h_k^u)=0.
 \label{eq:first-biharmonic-replacement}
\end{align}
In other words, $G_{k,R}$ is $L_0$-biharmonic in $B_{2R}$. Notice that $\tilde v_k-G_{k,R}=e_k^v-(b_k\cdot x)e_k^u$, so that \eqref{eq:first-replacement-errors} and \eqref{eq:first-smallness} imply
\begin{equation}\label{eq:biharmonic-error}
\|\tilde v_k-G_{k,R}\|_{L^\infty(B_{2R})}\le C_Rr_k(1+|b_k|)\longrightarrow0.
\end{equation}
Notice that in {\bf Step 3} we obtained the local uniform boundedness in $L^2$ for the sequence $\tilde v_k$; that is,
$$
\|\tilde v_k\|_{L^2(B_{2R})}\leq C_R.
$$
The latter together with \eqref{eq:biharmonic-error} imply that $G_{k,R}$ is uniformly
bounded in $L^2(B_{2R})$. Remark \ref{rem:Hj}, with $m=1$ and matrix $\tilde A_k(0)$, applied on $B_R\subset B_{2R}$, gives, for every integer $j\ge1$,
\begin{equation*}
\|G_{k,R}\|_{H^j(B_R)}\le C_{R,j}\|G_{k,R}\|_{L^2(B_{2R})}\le C_{R,j}.
\end{equation*}
Choosing $j>n/2$, the compact Sobolev embedding and \eqref{eq:biharmonic-error} give compactness of $\tilde v_k$ in $C^0(B_R)$. For $\tilde u_k$, compactness follows directly from \eqref{eq:E1}.

Taking a diagonal subsequence over $R=1,2,\ldots$, we obtain
$$
 \tilde u_k\to U,\qquad \tilde v_k\to V \quad\text{locally uniformly and strongly in }L^2_\loc(\R^n).
$$
Notice that \eqref{eq:biharmonic-error} implies also that $G_{k,R}\to V$ uniformly in $B_R$. The equation for $U$ follows by passing to the limit in the weak equation for $\tilde u_k$. Let us prove that $V$ is biharmonic. For $\varphi\in C_c^\infty(B_R)$, \eqref{eq:first-biharmonic-replacement} gives
$$
 0=\int_{B_R}G_{k,R}L_{\tilde A_k(0)}^2\varphi
   \longrightarrow\int_{B_R}VL_{A_0}^2\varphi.
$$
Since $R$ is arbitrary, the limiting equations are
\begin{equation}\label{eq:first-limit-equations}
 L_{A_0}U=0,\qquad L_{A_0}^2V=0 \qquad\text{in }\R^n.
\end{equation}
The normalizations and the growth bounds also pass to the limit:
\begin{gather}
\int_{B_1}U^2=\int_{B_1}V^2=1,\qquad\int_{B_1}VU\ell=0\quad\forall\,\ell\in\mathcal P_1,
 \label{eq:first-limit-normalization}\\
\int_{B_R}U^2\le D R^{\log_2D},\qquad\|V\|_{L^2(B_R)}\le C_\alpha R^{1+\alpha}\|U\|_{L^2(B_R)}\quad \forall R\ge1.\notag
\end{gather}

{\bf Step 5: transfer of the nodal inclusion.}
This step is identical to {\bf Step 5} of the proof of Proposition \ref{prop:constant-excess}. We conclude with the same argument that
$$
 Z(U)\subseteq Z(V).
$$

{\bf Step 6: contradiction by polynomial Liouville theorems and Murdoch's division lemma.}

Apply the change $x=Sy$, where $S=A_0^{1/2}$. Then 
$\tilde U(y)=U(Sy)$ is harmonic and $\tilde V(y)=V(Sy)$ is biharmonic in $\R^n$. Moreover, they retain similar $L^2$-integral growth conditions, with the same exponents. In particular
\begin{equation}\label{eq:polynomialbound-tildeU2}
\|\tilde U\|_{L^2(B_R)}\leq CR^{\frac12\log_2D}.
\end{equation}
Theorem \ref{thm:app-liouville}, with $m=0$, therefore implies that $\tilde U$ is a harmonic polynomial $Q$. Set $N=\deg Q$. Here the growth exponent in Theorem \ref{thm:app-liouville} is $\gamma=(\log_2D-n)/2\ge0$ (recall that $D\geq 2^n$). In particular, $N\le\lfloor\gamma\rfloor$.

Set $c_+=\max\{1,\sqrt\Lambda\}$, $c_-=\min\{1,\sqrt\lambda\}$ and $\tilde c=c_+/c_-$.
The inclusions $SB_R\subset B_{c_+R}$ and $S^{-1}B_{c_+R}\subset B_{\tilde cR}$, together with change of variables, give for any $R\geq1$
\begin{equation}\label{eq:changed-relative-growth2}
 \|\tilde V\|_{L^2(B_R)}\le C R^{1+\alpha}\|\tilde U\|_{L^2(B_{\tilde cR})}.
\end{equation}
Here $C$ depends only on the original growth constant, $\alpha,\lambda,\Lambda$.
Since $\tilde U=Q$ is a polynomial of degree $N$, its last norm is bounded by $CR^{N+n/2}$, with $C$ independent of $R$. Then,
\begin{equation*}
\|\tilde V\|_{L^2(B_{R})}\leq CR^{1+\alpha+N+\frac{n}{2}}.
\end{equation*}
Theorem \ref{thm:app-liouville} with $m=1$ implies that $\tilde V$ is a biharmonic polynomial $P$ of degree $\deg P\leq\lfloor 1+\alpha+N\rfloor$. Since $\alpha\in(0,1)$, then $\lfloor 1+\alpha+N\rfloor=1+N$. Hence $\deg P\leq\deg Q +1$.

The applied linear change of coordinates does preserve the zero set inclusion; that is, $Z(Q)\subseteq Z(P)$. Since $Q$ is a nontrivial harmonic polynomial, Murdoch's division Lemma \ref{lem:murdoch} yields $P=QR$ for some polynomial $R$. The degree inequality $\deg P\leq\deg Q+1$ implies that $R$ is affine; that is, $R\in\mathcal P_1$.

Returning to the original coordinates yields $V=U\tilde R$ in $\R^n$, for some $\tilde R\in\mathcal P_1$. Indeed, $\tilde R=R\circ S^{-1}$. Notice that the composition with an invertible linear map preserves both polynomiality and polynomial degree. The orthogonality and the normalizations in \eqref{eq:first-limit-normalization} therefore give a contradiction
$$
 0=\int_{B_1}VU\tilde R=\int_{B_1}V^2=1.
$$
The contradiction bounds $\Theta(\rho)$ uniformly as $\rho\to0$. Restoring the independent normalizations of $u$ and $v$ proves \eqref{eq:affine-excess-bound}.
\end{proof}

\subsection{Proof of Theorem \ref{thm:lip}}\label{sec:lip-completion}
We finally complete the proof of the $C^{1,\alpha}$ estimate.
\begin{proof}[Proof of Theorem \ref{thm:lip}]
Set $r_0=1/4$ and $K=C_\alpha\mathcal N(u,v)$, so that Proposition \ref{prop:affine-excess} gives
$$
\mathcal E_1(x_0,r)\leq Kr^{1+\alpha}\qquad\forall\,x_0\in B_1,\quad 0<r\leq r_0.
$$
Let
$$
r_j=2^{-j}r_0,\qquad j\in\mathbb N.
$$
By Lemma \ref{lem:affine-comparison} point $ii)$,
$$
|a_{x_0,r_j}-a_{x_0,r_{j+1}}|+r_j|b_{x_0,r_j}-b_{x_0,r_{j+1}}|\leq CKr_j^{1+\alpha}.
$$
Therefore $\{a_{x_0,r_j}\}_{j\in\mathbb N}$, $\{b_{x_0,r_j}\}_{j\in\mathbb N}$ are Cauchy sequences and we may define
$$
a(x_0)=\lim_{j\to\infty}a_{x_0,r_j},\qquad b(x_0)=\lim_{j\to\infty}b_{x_0,r_j}.
$$
Summing the corresponding geometric series gives
\begin{equation}\label{eq:affine-to-limit}
\begin{aligned}
|a(x_0)-a_{x_0,r_j}|&\le CK\sum_{l=j}^\infty r_l^{1+\alpha}=CKr_j^{1+\alpha}\sum_{m=0}^\infty (2^{-(1+\alpha)})^m\le CKr_j^{1+\alpha},\\
|b(x_0)-b_{x_0,r_j}|&\le CK\sum_{l=j}^\infty r_l^\alpha=CKr_j^\alpha\sum_{m=0}^\infty (2^{-\alpha})^m\le CKr_j^\alpha.
\end{aligned}
\end{equation}

Notice that $a(x_0)=w(x_0)$: indeed, by \eqref{eq:fit-size} with the choice $\beta=\alpha$, we know that
$$
 |a_{x_0,r_j}-w(x_0)|\le CKr_j^\alpha.
$$

Let $x_0,y_0\in B_1$ and set
$h=|x_0-y_0|$. Assume first that $0<h\leq r_0/2$, and choose $j\geq1$
such that $r_{j+1}<h\leq r_j$. We want to estimate the coefficients of
$$\ell_{y_0,r_j}(x)-\ell_{x_0,r_{j-1}}(x)=[a_{y_0,r_j}-a_{x_0,r_{j-1}}-b_{x_0,r_{j-1}}\cdot(y_0-x_0)]+(b_{y_0,r_j}-b_{x_0,r_{j-1}})\cdot(x-y_0),$$
seen as an affine polynomial centered at $y_0$.

Notice that $B_{r_j}(y_0)\subset B_{2r_j}(x_0)=B_{r_{j-1}}(x_0)$. Applying the lower bound in \eqref{eq:affine-moments} to the difference of the two affine polynomials, viewed as an affine polynomial centered at $y_0$, and then using Minkowski's inequality, we obtain

\begin{align*}
 \sqrt{c_0(|a_{y_0,r_j}-a_{x_0,r_{j-1}}-b_{x_0,r_{j-1}}\cdot(y_0-x_0)|^2+r_j^2|b_{y_0,r_j}-b_{x_0,r_{j-1}}|^2)}\\
 \le\left(\fint_{B_{r_j}(y_0)}|\ell_{y_0,r_j}-\ell_{x_0,r_{j-1}}|^2 \,d\mu\right)^{1/2}\\
 \le \mathcal E_1(y_0,r_j)
   +\left(\frac{\mu(B_{r_{j-1}}(x_0))}{\mu(B_{r_j}(y_0))}\right)^{1/2}
       \mathcal E_1(x_0,r_{j-1}).    
\end{align*}
By the excess estimate and the doubling inequality, since $B_{r_{j-1}}(x_0)=B_{2r_j}(x_0)\subset B_{4r_j}(y_0)$, we have
$$
|a_{y_0,r_j}-a_{x_0,r_{j-1}}-b_{x_0,r_{j-1}}\cdot(y_0-x_0)|+r_j|b_{y_0,r_j}-b_{x_0,r_{j-1}}|\leq CKr_j^{1+\alpha}.
$$
Then, by $r_j<2h$, we have
\begin{equation}\label{eq:cross-affine}
|a_{y_0,r_j}-a_{x_0,r_{j-1}}-b_{x_0,r_{j-1}}\cdot(y_0-x_0)|\leq CKh^{1+\alpha},\qquad |b_{y_0,r_j}-b_{x_0,r_{j-1}}|\leq CKh^{\alpha}.
\end{equation}
First, with the first inequality in \eqref{eq:cross-affine}, we prove differentiability of $w$ at $x_0$. Using \eqref{eq:affine-to-limit} and \eqref{eq:cross-affine}, we have
\begin{align*}
|w(y_0)-w(x_0)-b(x_0)\cdot(y_0-x_0)|
&\leq
|w(y_0)-a_{y_0,r_j}|+\big|a_{y_0,r_j}-a_{x_0,r_{j-1}}-b_{x_0,r_{j-1}}\cdot(y_0-x_0)\big|\\
&\quad+|a_{x_0,r_{j-1}}-w(x_0)|+|b_{x_0,r_{j-1}}-b(x_0)|\,|y_0-x_0|\\
&\leq CK h^{1+\alpha}.
\end{align*}
Consequently, considering $y_0\to x_0$, and hence $h\to0^+$, we obtain that $w$ is differentiable at $x_0$ and $\nabla w(x_0)=b(x_0)$.

Then, the second inequality in \eqref{eq:cross-affine} allows us to prove H\"older continuity for $\nabla w$. In fact, applying also \eqref{eq:affine-to-limit}, we have
$$
|\nabla w(x_0)-\nabla w(y_0)|\leq |b(x_0)-b_{x_0,r_{j-1}}|+|b(y_0)-b_{y_0,r_j}|+|b_{y_0,r_j}-b_{x_0,r_{j-1}}|\leq CKh^{\alpha}.
$$

It remains to prove the $L^\infty$ bound for the gradient and to consider
pairs of points at distance larger than $r_0/2$.

By \eqref{eq:fit-size}, applied at the fixed scale $r_0$ with choice $\beta=\alpha$, we have
$$
r_0|b_{x_0,r_0}|\leq CKr_0^\alpha,
$$
and therefore, applying \eqref{eq:affine-to-limit} for $j=0$, we obtain
$$
|\nabla w(x_0)|\leq|b(x_0)-b_{x_0,r_0}|+|b_{x_0,r_0}|\leq CK.
$$
Hence
\begin{equation}\label{eq:gradient-linfty}
\|\nabla w\|_{L^\infty(B_1)}\leq CK.
\end{equation}
Recall that the $L^\infty$ bound for $w$ follows from the previously proved
H\"older estimate \eqref{eq:holder-scaled}. 

We have already proved that, whenever
$|x_0-y_0|\leq r_0/2$,
$$
|\nabla w(x_0)-\nabla w(y_0)|\leq CK|x_0-y_0|^\alpha.
$$
Assume now that $h>r_0/2$. Then, by \eqref{eq:gradient-linfty}, since $r_0$ is fixed
$$
|\nabla w(x_0)-\nabla w(y_0)|\leq 2\|\nabla w\|_{L^\infty(B_1)}\leq CK\leq CK|x_0-y_0|^\alpha.
$$
Then, summing up all the contributions, \eqref{eq:lip-conclusion} follows.
\end{proof}

Notice that, by a fixed rescaling and localization, the $C^{1,\alpha}$ estimate can be extended to a slightly larger interior ball; that is, of radius $1<r^*<3/2$. For later purposes, we choose $r^*=5/4$.

\begin{remark}\label{rem:lip-scaled}
Under the assumptions of Theorem \ref{thm:lip}, for every $0<\alpha<1$,
\begin{equation}\label{eq:lip-scaled}
 \|w\|_{C^{1,\alpha}(B_{5/4})}\le C(n,\lambda,\Lambda,L,D,\alpha)\mathcal N(u,v).
\end{equation}
\end{remark}

\section{$C^{k,\alpha}$ estimates under $C^{k-1,\alpha}$ coefficients for $k\geq2$}\label{sec:higher-order}
The aim of this section is the proof of Theorem \ref{thm:higher}. Throughout the section, $d\in\mathbb N_{\geq2}=\mathbb N\setminus\{0,1\}$ denotes both the maximal degree of the polynomial fit and the target regularity order; in other words, $d=k$ in the notation of Theorem \ref{thm:higher}. Remark \ref{rem:lip-scaled} gives in particular the estimate
\begin{equation}\label{eq:higher-initial-lipschitz}
 \|w\|_{C^{0,1}(B_{5/4})}\le C\mathcal N(u,v).
\end{equation}
This is the preliminary regularity needed for the higher-order argument.

Throughout Section \ref{sec:higher-order}, fix an auxiliary exponent $\eta\in(0,1)$ for the replacement estimates, for instance $\eta=1/2$.
Unless otherwise specified, constants depend only on $n,\lambda,\Lambda,L,D,d$ and the indicated exponents.

\subsection{Best polynomial approximation}\label{sec:polynomial-fits}
For $x_0\in B_{3/2}$ and $0<r\le1/4$, define the $d$th-order excess
\begin{equation}\label{eq:polynomial-excess}
 \mathcal E_d(x_0,r)=\inf_{\pi\in\mathcal P_d}\left(\fint_{B_r(x_0)}|w-\pi|^2 \,d\mu\right)^{1/2}=\inf_{\pi\in\mathcal P_d}\left(\frac{1}{\mu(B_r(x_0))}\int_{B_r(x_0)}|v-\pi u|^2\right)^{1/2},
\end{equation}
where $\mathcal P_d$ is the space of polynomials of degree at most $d$.

The next lemma provides existence and uniqueness of the weighted polynomial fit.
\begin{Lemma}\label{lem:polynomial-minimizer}
For every center $x_0\in B_{3/2}$ and every radius $0<r\le1/4$, the minimization problem \eqref{eq:polynomial-excess} has a unique minimizer $\pi_{x_0,r}$. Moreover, it is characterized by the normal equations
\begin{equation}\label{eq:polynomial-normal-equations}
 \int_{B_r(x_0)}(w-\pi_{x_0,r})\pi \,d\mu=0\qquad\text{for every }\pi\in\mathcal P_d.
\end{equation}
\end{Lemma}
\begin{proof}
As in Lemma \ref{lem:affine-minimizer}, $\mathcal P_d$ is a finite-dimensional subspace of $L^2(B_r(x_0),d\mu)$. Distinct polynomials represent distinct elements: by Lemma \ref{lem:mass} and continuity, $u$ is nonzero on a smaller open ball, and a polynomial whose weighted norm vanishes must vanish on that ball and hence everywhere. The subspace is closed, so the Hilbert-space projection theorem gives existence, uniqueness and the normal equations \eqref{eq:polynomial-normal-equations}.
\end{proof}

Then, we prove the uniform weighted norm equivalence for polynomial functions.

\begin{Lemma}\label{lem:polynomial-moments}
There are $c_d,C_d>0$, depending only on $n,d,\lambda,\Lambda,L,D$,
such that, for $x_0\in B_{3/2}$, $0<r\le1/4$ and $\pi\in\mathcal P_d$,
\begin{equation}\label{eq:polynomial-moments}
c_d\sum_{|\nu|\le d}r^{2|\nu|}|D^\nu\pi(x_0)|^2\le\fint_{B_r(x_0)}|\pi|^2\,d\mu\le C_d\sum_{|\nu|\le d}r^{2|\nu|}|D^\nu\pi(x_0)|^2.
\end{equation}
\end{Lemma}

\begin{proof}
The exact Taylor formula
\begin{equation}\label{eq:taylor-formula}
\pi(x)=\sum_{|\nu|\le d}\frac{D^\nu\pi(x_0)}{\nu!}(x-x_0)^\nu
\end{equation}
gives the upper bound. For the lower bound, let us rescale as in the proof of
Lemma \ref{lem:affine-moments}; that is,
\begin{equation*}
\kappa^2=\frac{1}{r^n}\int_{B_r(x_0)}u^2,\qquad u_{x_0,r}(y)=\frac{u(x_0+ry)}{\kappa},\qquad A_{x_0,r}(y)=A(x_0+ry).
\end{equation*}
The normalized denominator $u_{x_0,r}$ has $L^2(B_1)$ norm one, $L^2(B_2)$ norm at most $\sqrt D$, see \eqref{eq:boundsB1B2}, and uniform interior bounds from \eqref{eq:E1} (for instance, $u_{x_0,r}$ is uniformly bounded in $C^{1,\eta}(B_{3/2})$ for a chosen $\eta\in(0,1)$).
Set
$$
q_\nu:=\frac{r^{|\nu|}D^\nu\pi(x_0)}{\nu!},\qquad |\nu|\le d.
$$
After the change of variables $x=x_0+ry$, the desired lower bound in \eqref{eq:polynomial-moments} is equivalent, up to changing the constant $c_d$, to
$$
\int_{B_1}u_{x_0,r}^2(y)\left|\sum_{|\nu|\le d}q_\nu y^\nu\right|^2\geq c_d\sum_{|\nu|\le d}|q_\nu|^2.
$$
If $\sum_{|\nu|\le d}|q_\nu|^2=0$, there is nothing to prove. Otherwise, dividing all the coefficients by
$$
\sqrt{\sum_{|\nu|\le d}|q_\nu|^2},
$$
and still using the same notation $q_\nu$ for the new coefficients, it is enough to prove
$$
\int_{B_1}u_{x_0,r}^2(y)\left|\sum_{|\nu|\le d}q_\nu y^\nu\right|^2\geq c_d
$$
for every family of coefficients satisfying
$$
\sum_{|\nu|\le d}|q_\nu|^2=1.
$$
If the lower bound failed, we could find normalized $u_k$ and polynomials
$$
\pi_k(y)=\sum_{|\nu|\le d}q^k_\nu y^\nu\qquad\mathrm{with}\quad \sum_{|\nu|\le d}|q^k_\nu|^2=1,
$$
such that
$$
\int_{B_1}u_k^2\pi_k^2\longrightarrow0.
$$
Take subsequences for which $u_k$ converges uniformly on $B_1$ to
$u_\infty$, with $\int_{B_1}u_\infty^2=1$, and the coefficient vectors of $\pi_k$ converge to those of a polynomial $\pi_\infty$ with the same normalization; that is,
$$
\pi_\infty(y)=\sum_{|\nu|\le d}q^\infty_\nu y^\nu\qquad\mathrm{with}\quad \sum_{|\nu|\le d}|q^\infty_\nu|^2=1.
$$
Then $u_\infty\pi_\infty=0$ on $B_1$. The nontrivial continuous function $u_\infty$ is nonzero on an open ball. Hence $\pi_\infty$ vanishes identically, a contradiction.
\end{proof}

In the next result, we provide quantitative control of the coefficients of the best polynomial approximation, together with control of their variation across scales.

\begin{Lemma}\label{lem:polynomial-comparison}
The following estimates hold:
\begin{itemize}
\item[$i)$] for $x_0\in B_1$ and $0<r\le1/4$,
\begin{equation}\label{eq:polynomial-size}
|\pi_{x_0,r}(x_0)-w(x_0)|+\sum_{1\le|\nu|\le d}r^{|\nu|}|D^\nu\pi_{x_0,r}(x_0)|\le C\mathcal N(u,v)r,
\end{equation}
where $C=C(n,d,\lambda,\Lambda,L,D)>0$;
\item[$ii)$] for $x_0\in B_{3/2}$, $0<r\le1/4$ and $r\le s\le\min\{2r,1/4\}$,
\begin{equation}\label{eq:polynomial-fit-difference}
\sum_{|\nu|\le d}r^{|\nu|}|D^\nu(\pi_{x_0,r}-\pi_{x_0,s})(x_0)|\le C\bigl(\mathcal E_d(x_0,r)+\mathcal E_d(x_0,s)\bigr),
\end{equation}
where $C=C(n,d,\lambda,\Lambda,L,D)>0$;
\item[$iii)$] for every fixed $\rho\in(0,1/4]$, there is $C=C(n,D,\rho)>0$ such that, for $x_0\in B_{3/2}$ and $\rho\le r\le1/4$,
\begin{equation}\label{eq:uniform-bound-polynomial}
\mathcal E_d(x_0,r)\le C\mathcal N(u,v).
\end{equation}
\end{itemize}
\end{Lemma}

\begin{proof}
Let us prove point $i)$. By Lemma \ref{lem:polynomial-minimizer}, $\pi_{x_0,r}$ is the orthogonal projection of $w$ onto the space of polynomials $\mathcal P_d$ in $L^2(B_r(x_0),d\mu)$. Since constants belong to this space and the orthogonal projection is linear, the projection of $w-w(x_0)$ is $\pi_{x_0,r}-w(x_0)$. Hence, by the contraction property of orthogonal projections, and the Lipschitz bound of $w$ in \eqref{eq:higher-initial-lipschitz} (notice that $B_r(x_0)\subset B_{5/4}$),
$$
\fint_{B_r(x_0)}|\pi_{x_0,r}-w(x_0)|^2\,d\mu\le\fint_{B_r(x_0)}|w-w(x_0)|^2\,d\mu\le C\mathcal N(u,v)^2r^{2}.
$$
Applying the lower bound in \eqref{eq:polynomial-moments} of Lemma \ref{lem:polynomial-moments} to the polynomial $\pi_{x_0,r}-w(x_0)$
gives \eqref{eq:polynomial-size}.

For point $ii)$, we want to estimate the distance between two polynomial fits at different scales
\begin{align*}
 \sqrt{c_d\sum_{|\nu|\le d}r^{2|\nu|}|D^\nu(\pi_{x_0,r}-\pi_{x_0,s})(x_0)|^2}
 &\le\left(\fint_{B_r(x_0)}|\pi_{x_0,r}-\pi_{x_0,s}|^2 \,d\mu\right)^{1/2}\\
 &\le \mathcal E_d(x_0,r)
   +\left(\frac{\mu(B_{s}(x_0))}{\mu(B_r(x_0))}\right)^{1/2}
       \mathcal E_d(x_0,s)\\
 &\le \mathcal E_d(x_0,r)+\sqrt D\,\mathcal E_d(x_0,s).
\end{align*}
In the previous inequalities, we used the lower bound in \eqref{eq:polynomial-moments}, we added and subtracted $w$ in the integral, we used the Minkowski inequality, the assumption $r\leq s$ and the definition of the excess. Then, by the assumption $s\leq 2r$, and the doubling inequality from scale $2r$ to $r$, \eqref{eq:polynomial-fit-difference} follows.

For the last point $iii)$, choosing the competitor $\pi=0$ in the minimization \eqref{eq:polynomial-excess}, we get
\begin{equation*}
\mathcal E_d(x_0,r)\le\frac{\|w\|_{L^2(B_2,d\mu)}}{\mu(B_r(x_0))^{1/2}}.
\end{equation*}
Then, Lemma \ref{lem:mass} and $r\geq\rho$ give 
$$\mu(B_r(x_0))\geq\mu(B_\rho(x_0))\ge c_\rho\mu(B_1).$$
The latter gives \eqref{eq:uniform-bound-polynomial}.
\end{proof}

\subsection{Decay of the higher-order excess}\label{sec:poly-algebra}
The next result shows that the assumptions of Theorem \ref{thm:higher} imply a quantitative decay of the $d$th-order excess.

\begin{Proposition}\label{prop:higher-excess}
Fix $r_0=1/4$. The assumptions of Theorem \ref{thm:higher} (with notation $k=d$) imply that there is $C=C(n,d,\lambda,\Lambda,L,D,\alpha)>0$ such that for any $x_0\in B_1$ and any $0<r\le r_0$
\begin{equation}\label{eq:higher-excess-bound}
 \mathcal E_d(x_0,r)\le C\mathcal N(u,v)r^{d+\alpha}.
\end{equation}
\end{Proposition}
\begin{proof}
Without loss of generality $v$ is nontrivial in $B_2$. Replace $u$ by $u/\|u\|_{L^2(B_1)}$ and $v$ by $v/\|v\|_{L^2(B_2)}$. The new quotient is $w/\mathcal N(u,v)$ and the new excess is $\mathcal E_d/\mathcal N(u,v)$. The equations, zeros, and the doubling constant $D$ are unchanged. We prove the estimate for these normalized pairs, for which $\mu(B_1)=\|w\|_{L^2(B_2,d\mu)}^2=\|v\|_{L^2(B_2)}^2=1$, using the fixed radius $r_0=1/4$ throughout this proof. Remark \ref{rem:lip-scaled} gives \eqref{eq:higher-initial-lipschitz} in $B_{5/4}$. Every ball with center in $B_1$ and radius at most $r_0=1/4$ lies in $B_{5/4}$.\\

{\bf Step 1: the contradiction argument and the choice of the scale.}
Suppose that \eqref{eq:higher-excess-bound} fails for normalized triples. Set
$$
 \Theta(\rho)=\sup_{A,u,v}\sup_{x_0\in B_1}\sup_{\rho\le r\le r_0}\frac{\mathcal E_d(x_0,r)}{r^{d+\alpha}},\qquad 0<\rho\le r_0,
$$
where the first supremum is over all normalized triples satisfying the fixed bounds $n,\lambda,\Lambda,L,D,d,\alpha$. By Lemma \ref{lem:polynomial-comparison} point $iii)$, $\Theta(\rho)<\infty$ for every $\rho>0$, whereas failure of the estimate means $\Theta(\rho)\to\infty$ as $\rho\to0$. Choose triples $(A_k,u_k,v_k)$, $\rho_k\to0$,
$x_k\in B_1$ and $r_k\in[\rho_k,r_0]$ such that
$$
\frac{\mathcal E_d(x_k,r_k)}{r_k^{d+\alpha}}\ge\frac{\Theta(\rho_k)}2.
$$
Set $\sigma_k=\mathcal E_d(x_k,r_k)$ and $M_k=\sigma_k/r_k^{d+\alpha}$. Then $M_k\to\infty$ and for any $s\in[r_k,r_0]$
\begin{equation}\label{eq:higher-maximal}
\mathcal E_d(x_k,s)\le2M_ks^{d+\alpha}.
\end{equation}
As in the previous proofs, finiteness of $\Theta(\rho)$ for positive $\rho$ implies $r_k\to0$.

{\bf Step 2: blow-up sequences and normalization.}
Let $\pi_k=\pi_{x_k,r_k}$ be the selected best polynomial. Define
\begin{align*}
\kappa_k^2&=\int_{B_1}u_k^2(x_k+r_kx),\qquad \tilde A_k(x)=A_k(x_k+r_kx),\\
\tilde u_k(x)&=\frac{u_k(x_k+r_kx)}{\kappa_k},\qquad \tilde v_k(x)=\frac{(v_k-\pi_ku_k)(x_k+r_kx)}{\kappa_k\sigma_k}.
\end{align*}
These functions are defined on $\Omega_k=(B_2-x_k)/r_k$, which
contains $B_{1/r_k}$. We have $L_{\tilde A_k}\tilde u_k=0$ and
$Z(\tilde u_k)\subseteq Z(\tilde v_k)$.
The excess and the normal equations \eqref{eq:polynomial-normal-equations} give
\begin{equation}\label{eq:higher-normalization}
\int_{B_1}\tilde u_k^2=\int_{B_1}\tilde v_k^2=1,\qquad\int_{B_1}\tilde v_k\tilde u_k\pi=0\quad\forall\,\pi\in\mathcal P_d.
\end{equation}

{\bf Step 3: integral growth of the blow-up sequences.}
For $R=1$, the bounds below follow from \eqref{eq:higher-normalization}. Fix $R>1$. Since $r_k\to0$, one has $Rr_k\leq r_0$ for every $k\geq \overline k(R)$. Then, let's consider $J\in\mathbb N\setminus\{0\}$ such that $2^{J-1}<R\leq 2^J$, and define $s_i=2^ir_k$ for any $i=0,\dots,J-1$ and $s_J=Rr_k$.

Let $\pi_{k,i}=\pi_{x_k,s_i}$ be the best polynomial fit on $B_{s_i}(x_k)$. Since $s_i\le s_{i+1}\le \min\{2s_i,1/4\}$, by Lemma \ref{lem:polynomial-comparison} point $ii)$ (see \eqref{eq:polynomial-fit-difference}) and \eqref{eq:higher-maximal}, for any $|\nu|\le d$ and any $i=0,\dots,J-1$
$$
|D^\nu(\pi_{k,i}-\pi_{k,i+1})(x_k)|\le CM_ks_i^{d+\alpha-|\nu|}.
$$
For $y\in B_{Rr_k}(x_k)$ and $i=0,\dots,J-1$, using the Taylor formula \eqref{eq:taylor-formula} we have
$$
|\pi_{k,i}(y)-\pi_{k,i+1}(y)|\leq\sum_{|\nu|\le d}\frac{|D^\nu(\pi_{k,i}-\pi_{k,i+1})(x_k)|}{\nu!}|y-x_k|^{|\nu|}\leq CM_k\sum_{|\nu|\le d}(Rr_k)^{|\nu|}s_i^{d+\alpha-|\nu|}.
$$
Hence, recalling that $2^{J-1}<R$, summation as in \eqref{eq:geometricsum} gives
\begin{equation*}
 |\pi_{k,0}(y)-\pi_{k,J}(y)|\le C M_k\sum_{|\nu|\le d}(Rr_k)^{|\nu|}\sum_{i=0}^{J-1}s_i^{d+\alpha-|\nu|}\le C M_k(Rr_k)^{d+\alpha}.
\end{equation*}

Then, in $B_{Rr_k}(x_k)$, adding and subtracting $\pi_{k,J}u_k$, and by the definition of the excess at scale $Rr_k$, we get
\begin{eqnarray*}
\|v_k-\pi_{k,0}u_k\|_{L^2(B_{Rr_k}(x_k))}&\leq& \big(\mathcal E_d(x_k,Rr_k)+\sup_{B_{Rr_k}(x_k)}|\pi_{k,0}(y)-\pi_{k,J}(y)|\big)\|u_k\|_{L^2(B_{Rr_k}(x_k))}\\
&\leq& C M_k(Rr_k)^{d+\alpha}\|u_k\|_{L^2(B_{Rr_k}(x_k))}.
\end{eqnarray*}
After changing variables and dividing by $\kappa_k\sigma_k=\kappa_kM_kr_k^{d+\alpha}$, we have
\begin{equation}\label{eq:higher-growth}
\|\tilde v_k\|_{L^2(B_R)}\le C R^{d+\alpha}\|\tilde u_k\|_{L^2(B_R)}.
\end{equation}
The constant depends only on the fixed data and is independent of $R,k$.

As in the proof of Proposition \ref{prop:constant-excess}, iterating the doubling inequality $N_0=\lceil\log_2R\rceil$ times gives
\begin{equation}\label{eq:higher-denominator-growth}
 \int_{B_R}\tilde u_k^2\le D^{\lceil\log_2R\rceil}\le D R^{\log_2D}.
\end{equation}
Consequently $\tilde u_k,\tilde v_k$ are bounded in $L^2$ on every fixed ball.\\

{\bf Step 4: compactness and the limiting equations.}
Let $L_0\cdot=\mathrm{div}(\tilde A_k(0)\nabla\cdot)$. The aim of this step is to approximate $\tilde v_k$, on each fixed ball, by $L_0$-polyharmonic functions, and to deduce compactness of the sequence by this approximation procedure. We also gain the fact that the limit in the blow-up procedure is polyharmonic.

Compared with {\bf Step 4} in the proof of Proposition \ref{prop:affine-excess}, the situation is more involved, since we need to introduce several correction terms in order to exploit the Taylor expansion of the variable coefficients. It is convenient to further split this step into substeps.\\

\emph{Step 4.1: estimates for $\tilde u_k,\tilde v_k$ and for the polynomial fits.}
Fix $R\ge1$ and an exponent $\eta\in(0,1)$, for instance $\eta=1/2$. After taking a subsequence, assume that $\tilde A_k(0)\to A_0$. Throughout the construction, $k$ is large enough that $B_{4R}\subset\Omega_k$ and $r_k\le1$. All constants $C_R$ below are independent of $k$.

Let us consider the normalized polynomial fit minus its zero order term; that is,
$$
\tilde\pi_k(x)=\frac{\pi_k(x_k+r_kx)-\pi_k(x_k)}{\sigma_k},
$$
and set
$$
v_k^*=\tilde v_k+\tilde\pi_k\tilde u_k=\frac{(v_k-\pi_k(x_k)u_k)(x_k+r_kx)}{\kappa_k\sigma_k}.
$$
Both $\tilde u_k$ and $v_k^*$ are $L_{\tilde A_k}$-harmonic. By \eqref{eq:polynomial-size}, for any $1\le|\nu|\le d$
$$
r_k^{|\nu|}|D^\nu\pi_k(x_k)|\le Cr_k.
$$
Using also the Taylor formula \eqref{eq:taylor-formula} for $\pi_k(x_k+r_kx)$, we have
\begin{equation*}
\tilde\pi_k(x)=\sum_{1\leq|\nu|\le d}\frac{r_k^{|\nu|}D^\nu\pi_k(x_k)}{\sigma_k\nu!}x^\nu,
\end{equation*}
so that
\begin{equation}\label{eq:higher-polynomial-size}
\|\tilde\pi_k\|_{C^{1,\eta}(B_{4R})}\le C_R B_k,\qquad B_k=\frac{r_k}{\sigma_k}.
\end{equation}
The latter inequality easily follows by
\begin{equation*}
\|\tilde\pi_k\|_{L^\infty(B_{4R})}+\|\nabla\tilde\pi_k\|_{L^\infty(B_{4R})}+\|D^2\tilde\pi_k\|_{L^\infty(B_{4R})}\le C_R B_k,
\end{equation*}
and
\begin{equation*}
\frac{|\nabla\tilde\pi_k(x)-\nabla\tilde\pi_k(y)|}{|x-y|^\eta}\le \|D^2\tilde\pi_k\|_{L^\infty(B_{4R})}|x-y|^{1-\eta}\leq (8R)^{1-\eta}\|D^2\tilde\pi_k\|_{L^\infty(B_{4R})}\qquad \forall x,y\in B_{4R}\quad\mathrm{with} \ x\neq y.
\end{equation*}

The growth estimates in {\bf Step 3} and the identity for $v_k^*$ give
$$
 \|\tilde u_k\|_{L^2(B_{4R})}\le C_R,\qquad \|v_k^*\|_{L^2(B_{4R})}\le C_R(1+B_k).
$$
The interior estimate \eqref{eq:E1} therefore yields
\begin{equation}\label{eq:higher-exact-Schauder}
\|\tilde u_k\|_{C^{1,\eta}(B_{2R})}\le C_R,\qquad \|v_k^*\|_{C^{1,\eta}(B_{2R})}\le C_R(1+B_k).
\end{equation}
Notice that this is not enough by itself to obtain compactness, since $B_k$ may diverge. We shall instead construct suitable $L_0$-polyharmonic approximations of $v_k^*$ and $\tilde u_k$. These will yield an $L_0$-polyharmonic approximation of $\tilde v_k$ with error bounded by $C_R r_k^{d-1+\alpha}(1+B_k)$. This is sufficient, since, recalling that $\sigma_k=M_kr_k^{d+\alpha}$ and $B_k=r_k/\sigma_k$,
\begin{equation}\label{eq:higher-needed-accuracy}
r_k^{d-1+\alpha}(1+B_k)=r_k^{d-1+\alpha}+\frac1{M_k}\longrightarrow0.
\end{equation}

\emph{Step 4.2: estimates of the Taylor polynomials of the rescaled coefficients and of the remainders.}
In order to obtain the power $d-1+\alpha$ in the error term, we exploit the $C^{d-1,\alpha}$ regularity of the variable coefficients by separating their Taylor polynomial of order $d-1$ from the corresponding remainder. Consider
\begin{align}
\tilde A_k&=\tilde A_k(0)+T_k+E_k,\label{eq:TaylorA}\\
T_k(x)&=\sum_{1\le|\nu|\le d-1}\frac{r_k^{|\nu|}}{\nu!}D^\nu A_k(x_k)x^\nu.\notag
\end{align}

We want to prove that
\begin{equation}\label{eq:Taylor-norms}
\|T_k\|_{C^{0,\eta}(B_{2R})}\le C_R r_k,\qquad \|E_k\|_{C^{0,\eta}(B_{2R})}\le C_R r_k^{d-1+\alpha}.
\end{equation}
The latter inequalities easily follow by Taylor's theorem for $A_k$ and for its first derivatives, which gives
\begin{equation*}
\|T_k\|_{L^\infty(B_{2R})}+\|\nabla T_k\|_{L^\infty(B_{2R})}\le C_R r_k,\qquad \|E_k\|_{L^\infty(B_{2R})}+\|\nabla E_k\|_{L^\infty(B_{2R})}\le C_R r_k^{d-1+\alpha}.
\end{equation*}
Notice that, in the estimates for the gradients, differentiation with respect to the rescaled variable produces an additional factor $r_k$. Moreover, for every $x,y\in B_{2R}$ with $x\neq y$,
\begin{equation*}
\frac{|T_k(x)-T_k(y)|}{|x-y|^\eta}\le \|\nabla T_k\|_{L^\infty(B_{2R})}|x-y|^{1-\eta}\leq (4R)^{1-\eta} \|\nabla T_k\|_{L^\infty(B_{2R})},
\end{equation*}
and
\begin{equation*}
\frac{|E_k(x)-E_k(y)|}{|x-y|^\eta}\le \|\nabla E_k\|_{L^\infty(B_{2R})}|x-y|^{1-\eta}\leq (4R)^{1-\eta} \|\nabla E_k\|_{L^\infty(B_{2R})}.
\end{equation*}
This proves \eqref{eq:Taylor-norms}.\\

\emph{Step 4.3: the iterative correction scheme and gain of decay.}
At this point, we would like to apply a construction similar to the one in {\bf Step 4} of Proposition \ref{prop:affine-excess}; namely, to consider the $L_0$-harmonic replacements $h_k^u,h_k^v$ of $\tilde u_k$ and $v_k^*$ in $B_{2R}$, and to use the corresponding errors $e_k^u,e_k^v$ to construct an $L_0$-polyharmonic approximation $G_{k,R}$ of $\tilde v_k$.

However, as observed above, the possibly diverging quantity $B_k=r_k/\sigma_k$ has to be compensated by an error of order $r_k^{d-1+\alpha}$. A single $L_0$-harmonic replacement would only produce an error of order $r_k$, which is not sufficient to close the argument. The idea is therefore to iterate the correction procedure a finite number of times, gaining one additional power of $r_k$ at each step, until the
error reaches the order $r_k^{d-1+\alpha}$.

We give the construction for a generic weak solution $f\in H^1(B_{4R})$ of $L_{\tilde A_k}f=0$. We will then apply it to $f=\tilde u_k$ and $f=v_k^*$. Start with the frozen $L_0$-harmonic replacement:
$$
\begin{cases}
-L_0h_{k,0}[f]=0&\text{in }B_{2R},\\
h_{k,0}[f]=f&\text{on }\partial B_{2R}.
\end{cases}
$$
This is the initial approximation. For $i=1,\ldots,d-1$, define the $i$th corrected approximation recursively by
\begin{equation}\label{eq:higher-corrections}
\begin{cases}
-L_0h_{k,i}[f]=\div\bigl(T_k\nabla h_{k,i-1}[f]\bigr)&\text{in }B_{2R},\\
h_{k,i}[f]=f&\text{on }\partial B_{2R}.
\end{cases}
\end{equation}

Thus every corrected approximation has the same boundary values as $f$. All these problems have unique weak solutions by Lax-Milgram.

Let $e_{k,i}[f]=f-h_{k,i}[f]$ be the error terms. These errors have zero boundary values; that is, they belong to $H^1_0(B_{2R})$. Since, by \eqref{eq:TaylorA}, we have
$$
 -L_0f=\div\bigl(T_k\nabla f+E_k\nabla f\bigr),
$$
subtracting the equations gives
\begin{equation*}
\begin{cases}
 -L_0e_{k,0}[f]=\div\bigl((T_k+E_k)\nabla f\bigr),\\
 -L_0e_{k,i}[f]=\div\bigl(T_k\nabla e_{k,i-1}[f]+E_k\nabla f\bigr), & 1\le i\le d-1.
\end{cases}
\end{equation*}
The boundary estimate \eqref{eq:E2} and \eqref{eq:Taylor-norms} give for any $1\le i\le d-1$
\begin{align*}
 \|e_{k,0}[f]\|_{C^{1,\eta}(B_{2R})}&\le C_Rr_k\|f\|_{C^{1,\eta}(B_{2R})},\\
 \|e_{k,i}[f]\|_{C^{1,\eta}(B_{2R})}&\le C_Rr_k\|e_{k,i-1}[f]\|_{C^{1,\eta}(B_{2R})}+C_Rr_k^{d-1+\alpha}\|f\|_{C^{1,\eta}(B_{2R})}.
\end{align*}
These estimates imply successively that
$e_{k,i}[f],h_{k,i}[f]\in C^{1,\eta}(B_{2R})$, provided $f\in C^{1,\eta}(B_{2R})$. This condition is satisfied for
$f=\tilde u_k$ and $f=v_k^*$ by \eqref{eq:higher-exact-Schauder}. 

Notice that each iteration gains one factor $r_k$. More precisely, induction gives
$$
\|e_{k,i}[f]\|_{C^{1,\eta}(B_{2R})}\le C_R\bigl(r_k^{i+1}+r_k^{d-1+\alpha}\bigr)\|f\|_{C^{1,\eta}(B_{2R})},\qquad 0\le i\le d-1.
$$
After $d-1$ iterations, the first term is $r_k^d$, which is bounded by $r_k^{d-1+\alpha}$ because $r_k\le1$ and $\alpha<1$. Hence
\begin{equation}\label{eq:poly-approx-error}
\|e_{k,d-1}[f]\|_{C^{1,\eta}(B_{2R})}=\|f-h_{k,d-1}[f]\|_{C^{1,\eta}(B_{2R})} \le C_Rr_k^{d-1+\alpha}\|f\|_{C^{1,\eta}(B_{2R})}.
\end{equation}

\emph{Step 4.4: the polyharmonic equations for the corrections.}
Now we want to introduce the correct approximation sequence for $\tilde v_k$. Their difference from $\tilde v_k$ must tend to zero, and they must solve a polyharmonic equation with constant coefficients. This will ensure compactness for $\tilde v_k$, and a limit polyharmonic equation for its blow-up limit.

As a general fact, notice that derivatives commute with constant coefficients polyharmonic operators of any order, and multiplication by a polynomial still gives a solution of a polyharmonic operator with a possibly larger order. In other words, for any two integers $p\geq1$ and $a\geq0$
\begin{equation}\label{eq:poly-product}
L_0^p h=0,\quad\pi\in\mathcal P_a\quad\Longrightarrow\quad L_0^{p+a}(\pi h)=0.
\end{equation}

Indeed, the product rule gives
$$
L_0(\pi h)=\pi L_0h+2\tilde A_k(0)\nabla\pi\cdot\nabla h+(L_0\pi)h.
$$
At each application of $L_0$, either another factor $L_0$ acts on $h$, or at least one derivative acts on $\pi$. A term with more than $a$ of the latter choices vanishes. After $p+a$ applications, every remaining term therefore contains a derivative of $L_0^p h$ and also vanishes. The above computations are valid in the sense of distributions, so no additional regularity of the corrections is assumed.

Then, for any $0\le i\le d-1$
\begin{equation}\label{eq:poly-hi}
 L_0^{1+id}h_{k,i}[f]=0\qquad\mathrm{in \ }B_{2R}.
\end{equation}
The latter follows by a recursion. The case $i=0$ holds true since $h_{k,0}[f]$ is the $L_0$-harmonic replacement of $f$ in $B_{2R}$. Suppose the result true for some $0\le i\le d-2$ and prove it for $i+1$. Let's consider the equation \eqref{eq:higher-corrections} for $h_{k,i+1}[f]$; that is,
$$
-L_0h_{k,i+1}[f]=\div\bigl(T_k\nabla h_{k,i}[f]\bigr)\qquad\text{in }B_{2R}.
$$
The right-hand side is
$$
\sum_{a,b=1}^n\left((T_k)_{ab}\partial^2_{ab}h_{k,i}[f]+\partial_a(T_k)_{ab}\,\partial_bh_{k,i}[f]\right).
$$
Since $L_0$ has constant coefficients, derivatives commute with $L_0$. Hence, by the inductive hypothesis, every derivative of $h_{k,i}[f]$ appearing above is annihilated by $L_0^p$, where $p=1+id$. Moreover, the polynomial coefficients $(T_k)_{ab}$ and $\partial_a(T_k)_{ab}$ have degree at most $d-1$. Therefore, by \eqref{eq:poly-product},
$$
L_0^{p+d-1}\div\bigl(T_k\nabla h_{k,i}[f]\bigr)=0.
$$
Applying $L_0^{p+d-1}$ to the equation for $h_{k,i+1}[f]$ gives $L_0^{p+d}h_{k,i+1}[f]=0$. Since $p+d=1+(i+1)d$, this proves \eqref{eq:poly-hi} for $i+1$ and completes the induction.\\

\emph{Step 4.5: compactness of $\tilde v_k$.}
Let us apply the construction of the corrections to the sequences $\tilde u_k$ and $v_k^*$, and define
$$
G_{k,R}=h_{k,d-1}[v_k^*]-\tilde\pi_kh_{k,d-1}[\tilde u_k].
$$
As we proved in \eqref{eq:poly-hi}, the polyharmonic order of each $h_{k,d-1}$ is at most $1+d(d-1)$. By \eqref{eq:poly-product}, multiplication by $\tilde\pi_k$, whose degree is at most $d$, increases this bound by at most $d$. Therefore
\begin{equation}\label{eq:polyharmonic-dquadro}
L_{0}^{d^2+1}G_{k,R}=0\qquad\text{in }B_{2R}.
\end{equation}
Here $d^2+1$ is an upper bound for the power in order to get $L_0$-polyharmonicity, not necessarily the minimal power. On the other hand,
$$
\tilde v_k-G_{k,R}=e_{k,d-1}[v_k^*]-\tilde\pi_ke_{k,d-1}[\tilde u_k].
$$
Combining \eqref{eq:poly-approx-error}, \eqref{eq:higher-exact-Schauder}, \eqref{eq:higher-polynomial-size}, and \eqref{eq:higher-needed-accuracy}, we obtain
\begin{equation}\label{eq:poly-compact-error}
\|\tilde v_k-G_{k,R}\|_{L^\infty(B_{2R})}\le C_Rr_k^{d-1+\alpha}(1+B_k)\longrightarrow0.
\end{equation}

Notice that in {\bf Step 3} we obtained the local uniform boundedness in $L^2$ for the sequence $\tilde v_k$; that is,
$$
\|\tilde v_k\|_{L^2(B_{2R})}\leq C_R.
$$
The latter together with \eqref{eq:poly-compact-error} imply that $G_{k,R}$ is uniformly
bounded in $L^2(B_{2R})$. Remark \ref{rem:Hj}, with $m=d^2$ and matrix $\tilde A_k(0)$, applied on $B_R\subset B_{2R}$, gives, for every integer $j\ge1$,
\begin{equation*}
\|G_{k,R}\|_{H^j(B_R)}\le C_{R,j}\|G_{k,R}\|_{L^2(B_{2R})}\le C_{R,j}.
\end{equation*}
Choosing $j>n/2$, the compact Sobolev embedding and \eqref{eq:poly-compact-error} give compactness of $\tilde v_k$ in $C^0(B_R)$. For $\tilde u_k$, compactness follows directly from \eqref{eq:E1}.

Taking a diagonal subsequence over $R=1,2,\ldots$, we obtain
$$
 \tilde u_k\to U,\qquad \tilde v_k\to V \quad\text{locally uniformly and strongly in }L^2_\loc(\R^n).
$$
Notice that \eqref{eq:poly-compact-error} implies also that $G_{k,R}\to V$ uniformly in $B_R$. The equation for $U$ follows by passing to the limit in the weak equation for $\tilde u_k$. Let us prove that $V$ is polyharmonic. For $\varphi\in C_c^\infty(B_R)$, \eqref{eq:polyharmonic-dquadro} gives
$$
 0=\int_{B_R}G_{k,R}L_{\tilde A_k(0)}^{d^2+1}\varphi
   \longrightarrow\int_{B_R}VL_{A_0}^{d^2+1}\varphi.
$$
Since $R$ is arbitrary, the limiting equations are
\begin{equation}\label{eq:higher-limit-equations}
 L_{A_0}U=0,\qquad L_{A_0}^{d^2+1}V=0 \qquad\text{in }\R^n.
\end{equation}
The normalizations and the growth bounds also pass to the limit:
\begin{gather}
\int_{B_1}U^2=\int_{B_1}V^2=1,\qquad\int_{B_1}VU\pi=0\quad\forall\,\pi\in\mathcal P_d,
 \label{eq:higher-limit-normalization}\\
\int_{B_R}U^2\le D R^{\log_2D},\qquad\|V\|_{L^2(B_R)}\le C R^{d+\alpha}\|U\|_{L^2(B_R)}\quad \forall R\ge1.\notag
\end{gather}

{\bf Step 5: transfer of the nodal inclusion.}
This step is identical to {\bf Step 5} of the proof of Proposition \ref{prop:constant-excess}. We conclude with the same argument that
$$
 Z(U)\subseteq Z(V).
$$

{\bf Step 6: contradiction by polynomial Liouville theorems and Murdoch's division lemma.}
Apply the change $x=Sy$, where $S=A_0^{1/2}$. Then 
$\tilde U(y)=U(Sy)$ is harmonic and $\tilde V(y)=V(Sy)$ is polyharmonic (of order $d^2+1$) in $\R^n$. Moreover, they retain similar $L^2$-integral growth conditions, with the same exponents. In particular
\begin{equation}\label{eq:polynomialbound-tildeU3}
\|\tilde U\|_{L^2(B_R)}\leq CR^{\frac12\log_2D}.
\end{equation}
Theorem \ref{thm:app-liouville}, with $m=0$, therefore implies that $\tilde U$ is a harmonic polynomial $Q$. Set $N=\deg Q$. Here the growth exponent in Theorem \ref{thm:app-liouville} is $\gamma=(\log_2D-n)/2\ge0$ (recall that $D\geq 2^n$). In particular, $N\le\lfloor\gamma\rfloor$.

Set $c_+=\max\{1,\sqrt\Lambda\}$, $c_-=\min\{1,\sqrt\lambda\}$ and $\tilde c=c_+/c_-$.
The inclusions $SB_R\subset B_{c_+R}$ and $S^{-1}B_{c_+R}\subset B_{\tilde cR}$, together with change of variables, give for any $R\geq1$
\begin{equation}\label{eq:changed-relative-growth3}
 \|\tilde V\|_{L^2(B_R)}\le C R^{d+\alpha}\|\tilde U\|_{L^2(B_{\tilde cR})}.
\end{equation}
Here $C$ depends only on the original growth constant, $d,\alpha,\lambda,\Lambda$.
Since $\tilde U=Q$ is a polynomial of degree $N$, its last norm is bounded by $CR^{N+n/2}$, with $C$ independent of $R$. Then,
\begin{equation*}
\|\tilde V\|_{L^2(B_{R})}\leq CR^{d+\alpha+N+\frac{n}{2}}.
\end{equation*}
Theorem \ref{thm:app-liouville} with $m=d^2$ implies that $\tilde V$ is a polyharmonic (of order $d^2+1$) polynomial $P$ of degree $\deg P\leq\lfloor d+\alpha+N\rfloor$. Since $\alpha\in(0,1)$, then $\lfloor d+\alpha+N\rfloor=d+N$. Hence $\deg P\leq\deg Q +d$.

The applied linear change of coordinates does preserve the zero set inclusion; that is, $Z(Q)\subseteq Z(P)$. Since $Q$ is a nontrivial harmonic polynomial, Murdoch's division Lemma \ref{lem:murdoch} yields $P=QR$ for some polynomial $R$. The degree inequality $\deg P\leq\deg Q+d$ implies that $R$ has degree at most $d$; that is, $R\in\mathcal P_d$.

Returning to the original coordinates yields $V=U\tilde R$ in $\R^n$, for some $\tilde R\in\mathcal P_d$. Indeed, $\tilde R=R\circ S^{-1}$. Notice that the composition with an invertible linear map preserves both polynomiality and polynomial degree. The orthogonality and the normalizations in \eqref{eq:higher-limit-normalization} therefore give a contradiction
$$
 0=\int_{B_1}VU\tilde R=\int_{B_1}V^2=1.
$$
The contradiction bounds $\Theta(\rho)$ uniformly as $\rho\to0$. Restoring the independent normalizations of $u$ and $v$ proves \eqref{eq:higher-excess-bound}.
\end{proof}

\subsection{Proof of Theorem \ref{thm:higher}}\label{sec:higher-proof}
We finally complete the proof of the $C^{k,\alpha}$ estimate.

\begin{proof}[Proof of Theorem~\ref{thm:higher}]
Set $d=k$, where $k$ is the regularity index in the statement of the
theorem. Set $r_0=1/4$ and $K=C\mathcal N(u,v)$, so that Proposition \ref{prop:higher-excess} gives
$$
\mathcal E_d(x_0,r)\leq Kr^{d+\alpha}\qquad\forall\,x_0\in B_1,\quad 0<r\leq r_0.
$$
Let
$$
r_j=2^{-j}r_0,\qquad j\in\mathbb N.
$$
By Lemma \ref{lem:polynomial-comparison} point $ii)$, for any $|\nu|\leq d$
\begin{equation}\label{eq:higher-coefficient-increment}
\left|D^\nu\pi_{x_0,r_{j+1}}(x_0)-D^\nu\pi_{x_0,r_j}(x_0)\right|\le CKr_j^{d+\alpha-|\nu|}.
\end{equation}
Since $d+\alpha-|\nu|>0$, $\{D^\nu\pi_{x_0,r_j}(x_0)\}_{j\in\mathbb N}$ are Cauchy sequences and we may define
$$
a_\nu(x_0)=\lim_{j\to\infty}D^\nu\pi_{x_0,r_j}(x_0).
$$
Moreover, summing \eqref{eq:higher-coefficient-increment}, we have
\begin{align}\label{eq:general-jet-error}
\left|a_\nu(x_0)-D^\nu\pi_{x_0,r_j}(x_0)\right|&\le CK\sum_{l=j}^{\infty}r_l^{d+\alpha-|\nu|}\\
&=CKr_j^{d+\alpha-|\nu|}\sum_{m=0}^{\infty}(2^{-(d+\alpha-|\nu|)})^m\notag\\
&\leq CKr_j^{d+\alpha-|\nu|}.\notag
\end{align}
Notice that $a_0(x_0)=w(x_0)$: indeed, by \eqref{eq:polynomial-size} we know that
$$
 |\pi_{x_0,r_j}(x_0)-w(x_0)|\le CKr_j.
$$
We next prove a uniform bound for the limiting coefficients. By Lemma \ref{lem:polynomial-comparison} point $i)$, applied at the fixed scale $r_0$,
$$
|\pi_{x_0,r_0}(x_0)-w(x_0)|+\sum_{1\le|\nu|\le d}r_0^{|\nu|}|D^\nu\pi_{x_0,r_0}(x_0)|\le CKr_0.
$$
Since $r_0$ is fixed and $\|w\|_{L^\infty(B_1)}\le CK$, it follows that $|D^\nu\pi_{x_0,r_0}(x_0)|\le CK$ for any $|\nu|\leq d$. Combining this with \eqref{eq:general-jet-error}, with $j=0$, gives
\begin{equation}\label{eq:general-jet-linfty}
\sup_{x_0\in B_1}|a_\nu(x_0)|\le CK,\qquad |\nu|\le d.
\end{equation}

Let $x_0,y_0\in B_1$ and set $h=|x_0-y_0|$. Assume first that $0<h\leq r_0/2$, and choose $j\geq1$ such that $r_{j+1}<h\leq r_j$. We compare the polynomials at these two centres by setting
$$
\widehat\pi=\pi_{y_0,r_j}-\pi_{x_0,r_{j-1}}.
$$

Notice that $B_{r_j}(y_0)\subset B_{2r_j}(x_0)=B_{r_{j-1}}(x_0)$. Applying the lower bound in \eqref{eq:polynomial-moments} to $\widehat\pi$, and then using Minkowski's inequality, we obtain

\begin{align*}
\left(c_d\sum_{|\nu|\le d}r_j^{2|\nu|}|D^\nu \widehat\pi(y_0)|^2\right)^{1/2}
&\le\left(\fint_{B_{r_j}(y_0)}|\widehat\pi|^2\,d\mu\right)^{1/2}\\
&\le\mathcal E_d(y_0,r_j)+\left(\frac{\mu(B_{r_{j-1}}(x_0))}{\mu(B_{r_j}(y_0))}
\right)^{1/2}\mathcal E_d(x_0,r_{j-1}).
\end{align*}
By the excess estimate and the doubling inequality, since $B_{r_{j-1}}(x_0)=B_{2r_j}(x_0)\subset B_{4r_j}(y_0)$, we have
$$
\sum_{|\nu|\le d}r_j^{|\nu|}|D^\nu \widehat\pi(y_0)|\leq CKr_j^{d+\alpha}.
$$
Then, by $r_j<2h$, we have, for every $|\nu|\le d$,
\begin{equation}\label{eq:higher-cross-coeff}
\left|D^\nu\pi_{y_0,r_j}(y_0)-D^\nu\pi_{x_0,r_{j-1}}(y_0)\right|\le CKh^{d+\alpha-|\nu|}.
\end{equation}
We only need the case $\nu=0$ of this estimate. By the exact Taylor formula,
\begin{equation}\label{eq:T1}
 \pi_{x_0,r_{j-1}}(y_0)=\sum_{|\nu|\le d}\frac{D^\nu\pi_{x_0,r_{j-1}}(x_0)}{\nu!}(y_0-x_0)^\nu.
\end{equation}
For every $x_0\in B_1$, define the polynomial
\begin{equation}\label{eq:T2}
 \pi_{x_0}(x)=\sum_{|\nu|\le d}\frac{a_\nu(x_0)}{\nu!}(x-x_0)^\nu.
\end{equation}

To compare $w(y_0)$ with $\pi_{x_0}(y_0)$, add and subtract the two fits $\pi_{y_0,r_j},\pi_{x_0,r_{j-1}}$ at $y_0$ and use that $a_0(y_0)=w(y_0)$
\begin{eqnarray*}
|w(y_0)-\pi_{x_0}(y_0)|\le|a_0(y_0)-\pi_{y_0,r_j}(y_0)|+|\pi_{y_0,r_j}(y_0)-\pi_{x_0,r_{j-1}}(y_0)|+|\pi_{x_0,r_{j-1}}(y_0)-\pi_{x_0}(y_0)|.
\end{eqnarray*}
The first two terms are bounded by $CKr_j^{d+\alpha}$ and $CKh^{d+\alpha}$, respectively, by \eqref{eq:general-jet-error} and \eqref{eq:higher-cross-coeff}. For the last term, the two polynomials are expanded about the same point $x_0$. Thus the Taylor formulas \eqref{eq:T1} and \eqref{eq:T2} give
\begin{align*}
 |\pi_{x_0,r_{j-1}}(y_0)-\pi_{x_0}(y_0)|&\le \sum_{|\nu|\le d}\frac{|D^\nu\pi_{x_0,r_{j-1}}(x_0)-a_\nu(x_0)|}{\nu!}|(y_0-x_0)^\nu|\\
&\le CK\sum_{|\nu|\le d}r_{j-1}^{d+\alpha-|\nu|}h^{|\nu|}\le CKh^{d+\alpha}.
\end{align*}
Here we used $|(y_0-x_0)^\nu|\le h^{|\nu|}$, $r_{j-1}=2r_j$, $r_j<2h$, and $d+\alpha-|\nu|>0$. Combining these bounds yields
$$
 |w(y_0)-\pi_{x_0}(y_0)|\le CKh^{d+\alpha}.
$$

If $h>r_0/2$, the same estimate follows from \eqref{eq:general-jet-linfty} and $\|w\|_{L^\infty(B_1)}\le CK$, since $r_0$ is fixed and $h<2$. Consequently,
\begin{equation}\label{eq:general-Whitney}
 |w(y_0)-\pi_{x_0}(y_0)|\le CK|x_0-y_0|^{d+\alpha}\qquad\forall\,x_0,y_0\in B_1.
\end{equation}
By the polynomial characterization of H\"older spaces \cite[Appendix A, Remark A.1]{FerRos22}, \eqref{eq:general-Whitney} implies that $w\in C^{d,\alpha}(B_1)$,
with
$$
\|w\|_{C^{d,\alpha}(B_1)}\le C\bigl(\|w\|_{L^\infty(B_1)}+K\bigr)\le C(n,\lambda,\Lambda,L,D,d,\alpha)\mathcal N(u,v).
$$
Since $d=k$, this is \eqref{eq:higher-conclusion}.
\end{proof}

\appendix
\section{Regularity estimates and a polynomial Liouville theorem for
polyharmonic functions}\label{app:polyharmonic}
This appendix is devoted to a summary of some classical regularity results that are used throughout the paper. For the reader's convenience, we state these results explicitly and, in some cases, provide a proof. The statements are formulated in the precise form needed in our arguments and are not intended to be either general or optimal.

\begin{remark}\label{rem:second-order-schauder}
This remark summarizes the classical second-order Schauder estimates used in the main proofs of the paper:

\smallskip\noindent\textit{Interior estimate with variable coefficients.}
Let $0<r<R$, and suppose that $A$ is symmetric, uniformly elliptic with constants $0<\lambda\leq\Lambda$ in $B_R$, with
$\|A\|_{C^{0,\eta}(B_R)}\le L$ for some $0<\eta<1$. If
$u\in H^1(B_R)$ satisfies $L_Au=0$, then there exists a constant $C=C(n,\lambda,\Lambda,L,\eta,r,R)>0$ such that
\begin{equation}\label{eq:E1}
 \|u\|_{C^{1,\eta}(B_r)}\le C\|u\|_{L^2(B_R)}.
\end{equation}
Estimate \eqref{eq:E1} follows from the interior Schauder estimate for weak solutions, see \cite{GilTru01}. In particular, Lipschitz coefficients give this estimate for every $0<\eta<1$, with constants controlled by the ellipticity, the Lipschitz bound, $\eta$, $r$ and $R$.

\smallskip\noindent\textit{Boundary estimate for problems with frozen coefficients.}
Let $A_0$ be a constant symmetric uniformly elliptic matrix, with ellipticity constants $0<\lambda\leq\Lambda$. Let $R>0$, and let
$F\in C^{0,\eta}(B_R;\R^n)$ for some $0<\eta<1$.
The unique weak solution $u\in H^1_0(B_R)$ of
$$
 \begin{cases}
 -L_{A_0}u=\div F &\mathrm{in \ }B_R,\\
 u=0 &\mathrm{on \ }\partial B_R,
\end{cases}
$$
belongs to $C^{1,\eta}(B_R)$ and there exists a constant $C=C(n,\lambda,\Lambda,\eta,R)>0$ such that
\begin{equation}\label{eq:E2}
 \|u\|_{C^{1,\eta}(B_R)}\le C\|F\|_{C^{0,\eta}(B_R)}.
\end{equation}
Existence and uniqueness follow from Lax-Milgram. This is the classical global divergence-form Schauder estimate on a ball, see \cite{GilTru01}.
\end{remark}

The next result gives smoothness of polyharmonic functions, namely solutions to $(-\Delta)^{m+1}u=0$ for some $m\in\mathbb N$. It is a standard consequence of the hypoellipticity of higher order elliptic differential operators; see, for instance, \cite{Hor90}. For the reader's convenience, we include an elementary proof of the slightly more general statement with a smooth right-hand side $f$, the polyharmonic case corresponding to $f=0$.

\begin{Lemma}\label{lem:app-smoothness}
Let $m\in\mathbb N$, $R>0$ and let $f\in C^\infty(B_R)$. If $u\in\mathcal D'(B_R)$ satisfies
$$
(-\Delta)^{m+1}u=f\qquad\hbox{in }B_R
$$
in the sense of distributions, then $u$ is $C^\infty(B_R)$.
\end{Lemma}

\begin{proof}
We argue by induction on $m\in\mathbb N$. If $m=0$, then
$$
-\Delta u=f
\qquad\hbox{in }B_R,
$$
in the sense of distributions, and the conclusion follows from the classical Weyl lemma for the Poisson equation. Now let $m\ge1$, and assume the statement has been proved with $m-1$ and consider the case $m$; that is,
$$
(-\Delta)^{m+1}u=f
\qquad\hbox{in }B_R.
$$
Set $v=-\Delta u$. Then $v\in\mathcal D'(B_R)$ and
$$
(-\Delta)^m v=f
\qquad\hbox{in }B_R.
$$
By the inductive hypothesis, $v\in C^\infty(B_R)$. Then
$$
-\Delta u=v
\qquad\hbox{in }B_R.
$$
Again, the Weyl lemma for the Poisson equation gives $u\in C^\infty(B_R)$.
\end{proof}

The next result provides an iterated Caccioppoli inequality for $L_{A_0}$-polyharmonic functions. This is classical, but we provide a proof.

\begin{Lemma}\label{lem:app-caccioppoli}
Let $m\in\mathbb N$, let $A_0$ be a constant and symmetric uniformly elliptic matrix, with ellipticity constants $0<\lambda\leq\Lambda$, and let $u\in L^2(B_R)$ satisfy
\begin{equation}\label{eq:app-equation}
 (-L_{A_0})^{m+1}u=0\qquad\hbox{in }B_R
\end{equation}
in distributions. For every $j\in\mathbb N\setminus\{0\}$ there exists a constant $C_j=C(n,m,j,\lambda,\Lambda)>0$ such that
\begin{equation}\label{eq:poly-Caccioppoli}
 \int_{B_{2^{-j}R}}|D^ju|^2 \le \frac{C_j}{R^{2j}}\int_{B_R}u^2.
\end{equation}
Here $D^ju$ stands for any derivative $D^\nu u$ with $|\nu|=j$.
\end{Lemma}
\begin{proof}
By the linear change $x=A_0^{1/2}y$ and Lemma \ref{lem:app-smoothness}, applied locally in the transformed domain, $u$ is smooth in $B_R$. Notice that the operator $(-L_{A_0})^{m+1}$ is a constant-coefficient divergence-form elliptic operator of order $2(m+1)$: its symbol is $(A_0\xi\cdot\xi)^{m+1}$, bounded between $\lambda^{m+1}|\xi|^{2(m+1)}$ and $\Lambda^{m+1}|\xi|^{2(m+1)}$. For each $0<r<R$, smoothness gives $u\in H^{m+1}(B_r)$. We can therefore apply the classical second-order Caccioppoli inequality if $m=0$, and \cite[Lemma 3.1]{BarHofMay19} if $m\ge1$, on this strictly smaller ball $B_r$ to obtain the existence of a constant $C=C(n,m,\lambda,\Lambda)>0$ such that
$$
 \int_{B_{r/2}}|\nabla u|^2\le \frac{C}{r^2}\int_{B_r}u^2.
$$
Every derivative solves the same equation, since $A_0$ is constant. Then, an iteration on the radii $r,r/2,\ldots,r/2^{j-1}$ gives
$$
 \int_{B_{2^{-j}r}}|D^ju|^2\le C^j2^{j(j-1)}r^{-2j}\int_{B_r}u^2.
$$
Since the above constant does not depend on $r,R$, letting $r\to R^-$ proves \eqref{eq:poly-Caccioppoli}.
\end{proof}

The following is a direct consequence of Lemma \ref{lem:app-caccioppoli}.
\begin{remark}\label{rem:Hj}
Under the assumptions in Lemma \ref{lem:app-caccioppoli}, for every $j\in\mathbb N$ and $0<\rho<R$, there exists a constant $C=C(n,m,j,\lambda,\Lambda,\rho,R)>0$ such that
$$
 \|u\|_{H^j(B_\rho)}\le C\|u\|_{L^2(B_R)}.
$$
Indeed, apply \eqref{eq:poly-Caccioppoli} on translated balls of radius smaller than $R-\rho$ and cover $B_\rho$ by finitely many of their smaller concentric balls; do this for every derivative order at most $j$.
\end{remark}

The last result we present here is the classical polynomial Liouville theorem for entire polyharmonic functions. It states that any entire polyharmonic function with a polynomial $L^2$-growth rate bound is in fact a polynomial. The proof relies only on Lemma \ref{lem:app-caccioppoli}.

\begin{Theorem}\label{thm:app-liouville}
Let $m\in\mathbb N$, and let $u\in L^2_{\loc}(\R^n)$ satisfy
\[
 (-\Delta)^{m+1}u=0\qquad\hbox{in }\R^n
\]
in distributions. Suppose that there are constants $M>0$ and $\gamma\ge0$
such that
\begin{equation}\label{eq:app-polynomial-growth}
 \int_{B_R}u^2\le M R^{n+2\gamma}
 \qquad\hbox{for every }R\ge1.
\end{equation}
Then $u$ is a polynomial of degree at most $\lfloor\gamma\rfloor$.
\end{Theorem}
\begin{proof}
From Lemma \ref{lem:app-smoothness}, $u$ is smooth.
For every integer $j\ge1$, Lemma \ref{lem:app-caccioppoli} with $A_0=\mathbb I$ gives
$$
 \int_{B_{2^{-j}R}}|D^ju|^2
 \le C_j R^{-2j}\int_{B_R}u^2
 \le C_j M R^{n+2\gamma-2j}.
$$
Choose $j>\gamma+n/2$. Let $K\subset\R^n$ be any compact set. For $R$ sufficiently large, $K\subset B_{2^{-j}R}$, and therefore, letting $R\to\infty$
$$
\int_{K}|D^ju|^2
 \le C_j M R^{n+2\gamma-2j}\to0.
$$
Since $K$ is arbitrary, all derivatives of order $j$ vanish identically in $\R^n$. Thus $u$ is a polynomial. The conclusion now is immediate when $u\equiv0$. If $u\not\equiv0$, let $N=\deg u$ and let $u_N$ be its highest
homogeneous part. Then
$$
 \frac{1}{R^{n+2N}}\int_{B_R}u^2
 =\int_{B_1}\bigl(R^{-N}u(Rx)\bigr)^2
 \longrightarrow\int_{B_1}u_N^2>0.
$$
Consequently, there exist $R_0\geq1$ and $c>0$ such that
$$
\int_{B_R}u^2\geq cR^{n+2N}\qquad\hbox{for every }R\geq R_0.
$$
Combining this estimate with \eqref{eq:app-polynomial-growth}, we obtain
$$
cR^{n+2N}\leq MR^{n+2\gamma}
\qquad\hbox{for every }R\geq R_0,
$$
which forces $N\leq\gamma$. Since $N$ is an integer, $N\leq\lfloor\gamma\rfloor$.
\end{proof}

\section*{Acknowledgment}
The author is a member of the GNAMPA research group of INdAM-Istituto Nazionale di Alta Matematica and partially supported by the INdAM-GNAMPA 2026 project "Struttura fine e regolarit\`a in problemi variazionali non-lineari" (CUP E53C25002010001).

\section*{AI disclosure}
ChatGPT assisted in developing some of the mathematical arguments and in drafting and revising parts of the manuscript. The author reviewed and verified the resulting material and takes full responsibility for the contents of the paper.


\begin{thebibliography}{99}

\bibitem{Anc78}
A. Ancona, \textsl{Principe de Harnack \`a la fronti\`ere et th\'eor\`eme de Fatou pour un op\'erateur elliptique dans un domaine lipschitzien}, Ann. Inst. Fourier (Grenoble) {\bf 28} (4), (1978), 169--213.

\bibitem{BanBasBur91}
R. Ba\~nuelos, R. Bass, K. Burdzy, \textsl{H\"older domains and the boundary Harnack principle}, Duke Math. J. {\bf 64} (1), (1991), 195--200.


\bibitem{BarHofMay19}
A. Barton, S. Hofmann, S. Mayboroda, \textsl{Dirichlet and Neumann boundary values of solutions to higher order elliptic equations},
Ann. Inst. Fourier (Grenoble) {\bf 69} (4), (2019), 1627--1678.


\bibitem{BasBur91}
R. Bass, K. Burdzy, \textsl{A boundary Harnack principle in twisted H\"older domains}, Ann. Math. (2) {\bf 134} (2), (1991), 253--276.

\bibitem{Cam63}
S. Campanato, \textsl{Propriet\`a di h\"olderianit\`a di alcune classi di funzioni}, Ann. Scuola Norm. Sup. Pisa Cl. Sci. (3) {\bf 17}, (1963), 175--188.

\bibitem{Cam64}
S. Campanato, \textsl{Propriet\`a di una famiglia di spazi funzionali}, Ann. Scuola Norm. Sup. Pisa Cl. Sci. (3) {\bf 18}, (1964), 137--160.


\bibitem{Dah77}
B. Dahlberg, \textsl{Estimates of harmonic measure}, Arch. Ration. Mech. Anal. {\bf 65} (3), (1977), 275--288.



\bibitem{DesSav15}
D. De Silva, O. Savin, \textsl{A note on higher regularity boundary Harnack inequality}, Discrete Contin. Dyn. Syst. {\bf 35} (12), (2015), 6155--6163.

\bibitem{DesSav16}
D. De Silva, O. Savin, \textsl{Boundary Harnack estimates in slit domains and applications to thin free boundary problems}, Rev. Mat. Iberoam. {\bf 32} (3), (2016), 891--912.

\bibitem{DesSav20}
D. De Silva, O. Savin, \textsl{A short proof of boundary Harnack principle}, J. Differential Equations {\bf 269}, (2020), 2419--2429.

\bibitem{DonJeoVit24}
H. Dong, S. Jeon and S. Vita, \textsl{Schauder type estimates for degenerate or singular elliptic equations with DMO coefficients}, Calc. Var. Partial Differential Equations {\bf 63} (9), (2024), article 239.


\bibitem{Eva86}
L.~C. Evans, \textsl{Quasiconvexity and partial regularity in the calculus of variations}, Arch. Ration. Mech. Anal. {\bf 95} (3), (1986), 227--252.


\bibitem{EvaGar87}
L.~C. Evans, R.~F. Gariepy, \textsl{Blowup, compactness and partial regularity in the calculus of variations}, Indiana Univ. Math. J. {\bf 36} (2), (1987), 361--371.


\bibitem{FerRos22}
X. Fern\'andez-Real, X. Ros-Oton, \textsl{Regularity Theory for Elliptic PDE}, Zurich Lectures in Advanced Mathematics, EMS Press, 2022.

\bibitem{Fio26}
G. Fioravanti, \textsl{A priori regularity estimates for ratio of solutions to elliptic equations with a product structure of two-dimensional nodal sets}, (2026), preprint arXiv:2605.23666.


\bibitem{GarLin86}
N. Garofalo, F.-H. Lin, \textsl{Monotonicity properties of variational integrals, $A_p$ weights and unique continuation}, Indiana Univ. Math. J. {\bf 35} (2), (1986), 245--268.


\bibitem{GiaNec79}
M. Giaquinta, J. Ne\v{c}as, \textsl{On the regularity of weak solutions to nonlinear elliptic systems via Liouville's type property}, Comment. Math. Univ. Carolin. {\bf 20} (1), (1979), 111--121.

\bibitem{GilTru01}
D. Gilbarg and N. S. Trudinger, \textsl{Elliptic Partial Differential Equations of Second Order}, second edition, Classics in Mathematics, Springer, 2001.

\bibitem{GiuMir68}
E. Giusti, M. Miranda, \textsl{Sulla regolarit\`a delle soluzioni deboli di una classe di sistemi ellittici quasi-lineari}, Arch. Ration. Mech. Anal. {\bf 31}, (1968), 173--184.

\bibitem{Hor90}
L. H\"ormander, \textsl{The Analysis of Linear Partial Differential Operators I: Distribution Theory and Fourier Analysis}, second edition,
Grundlehren der mathematischen Wissenschaften {\bf 256}, Springer-Verlag, Berlin, 1990.


\bibitem{JeoVit24}
S. Jeon, S. Vita, \textsl{Higher order boundary Harnack principles in Dini type domains}, J. Differential Equations {\bf 412}, (2024), 808--856.


\bibitem{JerKen82}
D.~S. Jerison, C.~E. Kenig, \textsl{Boundary behavior of harmonic functions in non-tangentially accessible domains}, Adv. Math. {\bf 46}, (1982), 80--147.

\bibitem{Kem72}
J. Kemper, \textsl{A boundary Harnack principle for Lipschitz domains and the principle of positive singularities}, Comm. Pure Appl. Math. {\bf 25}, (1972), 247--255.


\bibitem{LinLin22}
F.-H. Lin, Z. Lin, \textsl{Boundary Harnack principle on nodal domains}, Sci. China Math. {\bf 65} (12), (2022), 2441--2458.


\bibitem{LogMal15}
A. Logunov, E. Malinnikova, \textsl{On ratios of harmonic functions}, Adv. Math. {\bf 274}, (2015), 241--262.


\bibitem{LogMal16}
A. Logunov, E. Malinnikova, \textsl{Ratios of harmonic functions with the same zero set}, Geom. Funct. Anal. {\bf 26} (3), (2016), 909--925.




\bibitem{Man14}
D. Mangoubi, \textsl{A gradient estimate for harmonic functions sharing the same zeros}, Electron. Res. Announc. Math. Sci. {\bf 21}, (2014), 62--71.



\bibitem{Mur64}
B.~H. Murdoch, \textsl{A theorem on harmonic functions}, J. Lond. Math. Soc. {\bf 39}, (1964), 581--588.



\bibitem{Sim97}
L. Simon, \textsl{Schauder estimates by scaling}, Calc. Var. Partial Differential Equations {\bf 5} (5), (1997), 391--407.


\bibitem{TerTorVit24}
S. Terracini, G. Tortone and S. Vita, \textsl{Higher order boundary Harnack principle via degenerate equations}, Arch. Ration. Mech. Anal. {\bf 248} (2), (2024), article 29.

\bibitem{TerTorVit26a}
S. Terracini, G. Tortone and S. Vita, \textsl{A priori H\"older estimates for equations degenerating on nodal sets}, to appear in Trans. Amer. Math. Soc. (2026), preprint arXiv:2507.18991. doi:10.1090/tran/9759.

\bibitem{TerTorVit26b}
S. Terracini, G. Tortone and S. Vita, \textsl{A priori regularity estimates for equations degenerating on nodal sets}, Adv. Math. {\bf 492}, (2026), article 110885.


\end{thebibliography}
\end{document}